\documentclass[10pt,a4paper]{article}
\usepackage{amssymb,amsmath,amsopn,amsthm,graphicx,makeidx,bbm,url}
\usepackage{mathrsfs}
\usepackage[colorlinks,linkcolor=blue,citecolor=blue,urlcolor=red]{hyperref}
\input xy
\xyoption{all}
\usepackage[all,2cell]{xy}
\UseAllTwocells

\usepackage{color}

\begin{document}

\newtheorem{definition}{Definition}[section]
\newtheorem{definitions}[definition]{Definitions}
\newtheorem{lemma}[definition]{Lemma}
\newtheorem{prop}[definition]{Proposition}
\newtheorem{theorem}[definition]{Theorem}
\newtheorem{cor}[definition]{Corollary}
\newtheorem{cors}[definition]{Corollaries}
\theoremstyle{remark}
\newtheorem{remark}[definition]{Remark}
\theoremstyle{remark}
\newtheorem{remarks}[definition]{Remarks}
\theoremstyle{remark}
\newtheorem{notation}[definition]{Notation}
\theoremstyle{remark}
\newtheorem{example}[definition]{Example}
\theoremstyle{remark}
\newtheorem{examples}[definition]{Examples}
\theoremstyle{remark}
\newtheorem{dgram}[definition]{Diagram}
\theoremstyle{remark}
\newtheorem{fact}[definition]{Fact}
\theoremstyle{remark}
\newtheorem{illust}[definition]{Illustration}
\theoremstyle{remark}
\newtheorem{rmk}[definition]{Remark}
\theoremstyle{definition}
\newtheorem{observation}[definition]{Observation}
\theoremstyle{definition}
\newtheorem{question}[definition]{Question}
\theoremstyle{definition}
\newtheorem{conj}[definition]{Conjecture}

\newcommand{\stac}[2]{\genfrac{}{}{0pt}{}{#1}{#2}}
\newcommand{\stacc}[3]{\stac{\stac{\stac{}{#1}}{#2}}{\stac{}{#3}}}
\newcommand{\staccc}[4]{\stac{\stac{#1}{#2}}{\stac{#3}{#4}}}
\newcommand{\stacccc}[5]{\stac{\stacc{#1}{#2}{#3}}{\stac{#4}{#5}}}

\renewcommand{\marginpar}[2][]{}

\renewenvironment{proof}{\noindent {\bf{Proof.}}}{\hspace*{3mm}{$\Box$}{\vspace{9pt}}}

\title{Model theory in compactly generated (tensor-)triangulated categories}

\author{Mike Prest and Rose Wagstaffe, \\ Department of Mathematics, University of Manchester, M13 9PL, UK \\
mprest@manchester.ac.uk \hspace{1cm}
rose.wagstaffe@manchester.ac.uk}

\date{}

\maketitle

\abstract{We give an account of model theory in the context of compactly generated triangulated and tensor-triangulated categories ${\cal T}$.  We describe pp formulas, pp-types and free realisations in such categories and we prove elimination of quantifiers and elimination of imaginaries.  We compare the ways in which definable subcategories of ${\cal T}$ may be specified.  Then we link definable subcategories of ${\cal T}$ and finite-type torsion theories on the category of modules over the compact objects of ${\cal T}$.  We briefly consider spectra and dualities.  If ${\cal T}$ is tensor-triangulated then new features appear, in particular there is an internal duality in rigidly-compactly generated tensor-triangulated categories.}

\footnotetext{Keywords:  triangulated category, tensor-triangulated category, model theory, definable subcategory}

\footnotetext{MSC:  03C60, 18E45, 18G80}

\tableofcontents

\section{Introduction and Background}

\subsection{Introduction}

Model theory in a compactly generated triangulated category ${\cal T}$ falls within the scope of the model theory of modules {\it via} the restricted Yoneda embedding ${\cal T} \to {\rm Mod}\mbox{-}{\cal T}^{\rm c}$ where ${\cal T}^{\rm c}$ denotes the subcategory of compact objects in ${\cal T}$.  The model theory of modules over possibly many-sorted rings, such as ${\cal T}^{\rm c}$, is well-developed but there are many special features of triangulated categories that make it worthwhile to directly develop model theory in the triangulated context.  That is what we do here, and we also consider additional features which appear when the category is tensor-triangulated.  A good number of the results appear elsewhere but we give a detailed and unified account which, we hope, will be a useful reference.

\vspace{4pt}

What began as the model theory of modules - the investigation of model-theoretic questions in the context of modules over a ring - has developed in scope - to much more general categories - in depth, and in purpose having for a long time been led by interests and questions coming from representation theory.  Many aspects - purity, pure-injectives, definable subcategories etc. - can be dealt with purely algebraically and, in the context of compactly generated triangulated categories, this was developed by Beligiannis and Krause \cite{BeligHomTriang}, \cite{KraTel} (for earlier relevant work, see \cite{ChrStr}, \cite{BenGnac}, and \cite{NeeRev}).  But, apart from a brief treatment in \cite{GarkPre2}, some use in \cite{ALPP} and a recent detailed exposition of some aspects in \cite{B-TSigmaPI}, there has not been much explicit appearance of model theory in triangulated categories.  To some extent that is because there is a `dictionary' between model theoretic and algebraic/functor-category methods, allowing much of what can be proved with model theory to be proved by other methods.  But sometimes what is obvious and natural using one language is not so easily translatable into the other.  Moreover, model theory can give new insights and simpler proofs.  Our main aim in this paper is to make the methods of model theory readily available to be used in compactly generated triangulated categories.  Some aspects - dualities, spectra, enhancements, extensions to well-generated triangulated categories - are currently in development, so we don't aim to be comprehensive but we do present the more settled material in detail.

\vspace{4pt}

Some minimal acquaintance with model theory, at least with basic ideas in the model theory of modules, will be helpful for the reader but we do keep formal aspects of model theory to a minimum.  Really, all that we need is the notion of a formula and its solution set in a structure.  

We do need to use sorted variables.  Variables in a formula are place-holders for elements from a structure; in our context these elements may belong to different {\em sorts}.  The idea is very simple and well-illustrated by representations of the quiver $A_2$ which is $\bullet \to \star$.  A representation of this quiver in the category of modules over a ring $R$ consists of two $R$-modules $M_\bullet$, $M_\star$ and an $R$-linear map from $M_\bullet$ to $M_\star$.  Such a structure is naturally two-sorted, with elements of the sort (labelled by) $\bullet$ being those of $M_\bullet$ and those of sort (labelled by) $\star$ being those of $M_\star$.  The variables we would use in writing formulas reflect that, say with subscripts, and for this example we would use variables of two sorts (labelled respectively by $\bullet$ and $\star$).  The difference between using a 2-sorted and 1-sorted language is the difference between treating (2-sorted) representations of that quiver (equivalently modules over the 2-sorted ring which is the ($R$-)path category of the quiver) and (1-sorted) modules over the path algebra of the quiver (the path algebra of the quiver is a normal, 1-sorted, ring).  That is a matter of choice if there are only finitely many sorts but, because ${\cal T}^{\rm c}$ is skeletally infinite, we do need to use sorted structures and take account of sorts in formulas.  For more discussion, and many examples, of this, see \cite{PreErice}.

\vspace{4pt}

{\bf We suppose throughout this paper that ${\cal T}$ is a compactly generated triangulated category.} We take this to include the requirement that ${\cal T}$ has infinite coproducts.  We suppose that the reader knows something about these categories, but we do recall here that the derived category ${\cal D}({\rm Mod}\mbox{-}R)$ of the category ${\rm Mod}\mbox{-}R$ of $R$-modules is a basic example which is obtained from the category of chain complexes of $R$-modules by a type of localisation process which preserves homological information.  The exact sequences of ${\rm Mod}\mbox{-}R$ give rise to {\em triangles} - certain triples of composable morphisms - in  ${\cal D}({\rm Mod}\mbox{-}R)$.  There is also a {\em shift} autoequivalence on  ${\cal D}({\rm Mod}\mbox{-}R)$ which is induced by the shift operation on chain complexes.  In general a {\em triangulated category} is an additive category equipped with a structure of triangles and a shift, subject to certain conditions which can be found in references such as \cite{NeeBk}, \cite[Chpt.~10]{Wei}, and \cite{StevTour} for tensor-triangulated categories.  

An object $A$ of a triangulated category ${\cal T}$ is {\bf compact} if the hom-functor $(A,-)$ commutes with direct sums and ${\cal T}$ is said to be {\bf compactly generated} if there is, up to isomorphism, just a set of compact objects in ${\cal T}$ and if the compact objects of ${\cal T}$ see every object in the sense that, if $X\in {\cal T}$ and if $(A,X)=0$ for every compact object $A$ in ${\cal T}$, then $X=0$.  The restriction that ${\cal T}$ be compactly generated could be weakened to ${\cal T}$ being well-generated but, in that case, model theory using infinitary languages would be needed, so we would lose the Compactness Theorem of model theory and its many consequences.  This is an interesting direction to follow and a start has been made, see \cite{KraLetz} for instance, but here we don't look any further in that direction (also cf.~\cite[\S 5B]{AdRo}).

Let ${\cal T}^{\rm c}$ denote the full subcategory of compact objects of ${\cal T}$.  Model theory for the objects of ${\cal T}$ is based on the key idea that the {\em elements} of objects of ${\cal T}$ are the morphisms from compact objects.  That is, if $X$ is an object of ${\cal T}$ and $A$ is a compact object of ${\cal T}$, then an {\bf element} of $X$ {\bf of sort} (indexed by) $A$ is a morphism $A\to X$ in ${\cal T}$, that is, the value of the functor $(-,X): ({\cal T}^{\rm c})^{\rm op} \to {\bf Ab}$ on $A$, where ${\bf Ab}$ denotes the category of abelian groups.  This is just an extension of the fact that, if $M$ is a (right) module over a (normal, 1-sorted) ring $R$, then the elements of $M$ may be identified with the morphisms from the module $R_R$ to $M$.

There is, up to isomorphism, just a set of compact objects, so we may use the objects in a small version of ${\cal T}^{\rm c}$ to index the sorts of the language for ${\cal T}$.  A ``small version" of ${\cal T}^{\rm c}$ means an equivalent category which has just a set of objects.  We don't go into detail about setting up the language - for that see \cite[Appx.~B]{PreNBK} or various other background references on the model theory of modules, for instance \cite[\S 5]{PreErice}, \cite[Chpt.~18]{PreMAMS} - because all we really need is that it gives us a way of writing down formulas, in particular (in our context) pp formulas.  Each formula defines, for every $X\in {\cal T}$, a certain subset of $(A_1,X) \oplus \dots \oplus (A_n,X)$ with $A_i\in {\cal T}^{\rm c}$ (the $A_i$ label the sorts of the free variables of the formula). 

Of course, for every object $X\in {\cal T}$, each sort $(A,X)$, for $A\in {\cal T}^{\rm c}$, has an abelian group structure, and this is built into the formal language.  Also built into the language is the action of (a small version of) ${\cal T}^{\rm c}$ on objects $X \in {\cal T}$ - the morphisms of ${\cal T}^{\rm c}$ ``multiply" the ``elements" of $X$, taking those of one sort to a possibly different sort.  Explicitly, if $f:A \to B$ is a morphism of ${\cal T}^{\rm c}$, then this induces $b\in (B,X) \mapsto bf \in (A,X)$ - multiplication by $f$ from sort $B$ to sort $A$.  Note how this generalises the action of a ring on a (1-sorted) right module.  In particular, each sort $(A,X)$ is a right module over ${\rm End}(A)$ but these multiplications on single sorts are only some of the multiplications that constitute the action of (the many-sorted ring) ${\cal T}^{\rm c}$ on objects $X$ of ${\cal T}$. 

In this way an object $X$ of ${\cal T}$ is replaced by a (many-sorted) set-with-structure, precisely by the right ${\cal T}^{\rm c}$-module which is the representable functor $(-,X)$ restricted to ${\cal T}^{\rm c}$.  This replacement is effected by the restricted Yoneda functor $y:{\cal T} \to {\rm Mod}\mbox{-}{\cal T}^{\rm c}$ which is given on objects by $X \to (-,X)\upharpoonright {\cal T}^{\rm c}$ and on morphisms $f:X\to Y$ by $f\mapsto (-,f):(-,X) \to (-,Y)$.  This functor is neither full nor faithful but, see \ref{ybijpi}, \ref{ybijcpct} below, it loses nothing of the model theory\footnote{That is because we use finitary model theory; infinitary languages would detect more, including some {\bf phantom} morphisms, that is, morphisms $f$ with $yf=0$.} so we may do model theory directly in ${\cal T}$ or, equivalently, we may move to the functor/module category ${\rm Mod}\mbox{-}{\cal T}^{\rm c}$, where the well-worked-out model theory of multisorted modules applies.  Sometimes it is more convenient to work in the one category than the other; in any case, moving from the one context to the other is straightforward (and is detailed in this paper).

The move to ${\rm Mod}\mbox{-}{\cal T}^{\rm c}$ gives us the immediate conclusion that the theory of ${\cal T}$ has pp-elimination of quantifiers.

\begin{theorem}\label{ppeq} \marginpar{ppeq}  If ${\cal T}$ is a compactly generated triangulated category, then every formula in the language for ${\cal T}$ is equivalent to the conjunction of a sentence (which refers to sizes of quotients of pp-definable subgroups) and a finite boolean combination of pp formulas.  
\end{theorem}

A pp formula (in our context) is an existentially quantified system of linear equations.  A system of $R$-linear equations over a possibly multisorted ring $R$ can be written in the form
$$\bigwedge_{j=1}^m\, \sum_{i=1}^n\, x_ir_{ij}=0_j
$$ 
(read the conjunction symbol $\bigwedge$ as ``and") \\
or, more compactly, as $\overline{x}G=0$ where $G =(r_{ij})_{ij}$ is a matrix over $R$.  Here $x_i$ is a variable of sort $i$ and $r_{ij}$ a morphism from sort $j$ to sort $i$ (we are dealing with right modules, hence the contravariance).  If we denote this (quantifierfree) formula as $\theta(\overline{x})$, that is, $\theta(x_1, \dots, x_n)$, then its solution set in a module $M$ is denoted $\theta(M)$ and is a subgroup of $M_1\oplus \dots \oplus M_n$ where $M_i$ is the group of elements of $M$ of sort $i$, that is $(-,\bullet_i)(M) \simeq M(\bullet_i)$ where $\bullet_i$ is the object of $R$ corresponding to sort $i$.

A projection of the solution set for such a system of equations is defined by a formula of the form 
$$\exists x_{k+1},\dots, x_n \, \big(\bigwedge_{j=1}^m\, \sum_{i=1}^n\, x_ir_{ij}=0_j\big).
$$  
A formula (equivalent to one) of this form is a {\bf pp} (for ``positive primitive") {\bf formula} (the term {\bf regular formula} also is used).  We can write a pp formula more compactly as $\exists \overline{y} \,\, (\overline{x} \,\, \overline{y}) G =0$, or $\exists \overline{y} \,\, (\overline{x} \,\, \overline{y}) \left( \begin{array}{c} G' \\ G'' \end{array} \right) =0$, equivalently $\exists \overline{y} \,\, \overline{x}G' = \overline{y}G''$, if we want to partition the matrix $G$.  If we denote this formula by $\phi(x_1, \dots, x_k)$ then its {\bf solution set} $\phi(M)$ in $M$ is the subgroup of $M_1\oplus \dots \oplus M_k$ obtained by projecting $\theta(M)$ to the first $k$ components.  We refer to such a solution set as a {\bf pp-definable subgroup} of $M$ (the terminologies ``subgroup of finite definition" and ``finitely matrizable subgroup" also have been used).

\begin{example}\label{ppex}\marginpar{ppex}  Consider the quiver $A_4$ with orientation shown $1\xrightarrow{\alpha} 2 \xleftarrow{\beta} 3 \xrightarrow{\gamma} 4$ and let $R=KA_4$ be its path algebra with coefficients from a field $K$.  So left $R$-modules, equivalently $K$-representations of $A_4$ have the shape $V_1 \xrightarrow{T_\alpha} V_2 \xleftarrow{T_\beta} V_3 \xrightarrow{T_\gamma} V_4$ where the $V_i$ are $K$-vectorspaces and $T_\alpha, T_\beta, T_\gamma$ are $K$-linear maps.  In order to illustrate the definitions above, we think of these structures as {\em right} modules over the {\em opposite} of the 4-object $K$-linear path category of $A_4$, that is, over the $K$-linear category which has objects $\bullet_i$, $i=1,2,3,4$, and with ${\rm End}(\bullet_i) = K\cdot 1_i$, $(\bullet_2,\bullet_1) = K\alpha$, $(\bullet_2,\bullet_3) = K\beta$, $(\bullet_4,\bullet_3) = K\gamma$ and all other morphism groups $0$.

The corresponding language has four sorts, and the function symbols are, apart from the additions in each sort, the $\lambda f$ where $\lambda\in K$ and $f$ is one of the identity maps or $\alpha$, $\beta$ or $\gamma$.  An example of a system of linear equations is 
$$x_2 -x_1\alpha - x_3\beta = 0_2$$
$$x_3\gamma =0_3$$
where sorts are shown by subscripts to variables and zeroes.  Note that all terms in a given equation must have the same sort.

We may quantify out the variables $x_1$ and $x_3$ to obtain the pp formula $\phi(x_2)$ which is
$$ \exists x_1, x_3\, (x_2 - x_1\alpha - x_3\beta = 0_2 \, \wedge\, x_3\gamma =0_3)$$
which, in matrix format is
$$\exists x_1, x_3 \,\,  \big(\begin{array}{ccc}x_1 & x_2 & x_3\end{array} \big)  \big(\begin{array}{cc} -\alpha & 0 \\ 1 & 0 \\ -\beta & \gamma \end{array} \big) =  \big(\begin{array}{cc} 0 & 0\end{array} \big).$$
The solution set $\phi(M)$ in any module $M$ is the set $\alpha(M)+\beta({\rm ker}_M(\gamma))$ - a $K$ vector subspace of $M\bullet_2$ ($=V_2$ in the representation-of-quivers notation).
\end{example}

All this applies to ${\cal T}$ since the model theory of ${\cal T}$ is essentially that of right ${\cal T}^{\rm c}$-modules.  So Theorem \ref{ppeq} follows because, if $R$ is a (possibly many sorted) ring, then the theory of $R$-modules has pp-elimination of quantifiers\footnote{For the formal statement see, for instance \cite[A.1.1]{PreNBK}.  That is given for 1-sorted modules but the general case reduces to this, see \cite[\S 1]{KP1}, because each formula involves only finitely many sorts, corresponding to $A_1, \dots, A_n$ say, so is equivalent to a formula over a 1-sorted ring, namely ${\rm End}(A_1\oplus \dots \oplus A_n)$.} and so this applies to the theory of the image of the restricted Yoneda embedding which, as we have remarked, is the theory of ${\cal T}$.

It turns out, \cite[3.1/3.2]{GarkPre2} and see Section \ref{secelimq}, that, with this language, the theory of ${\cal T}$ has complete (positive) elimination of quantifiers - every (pp) formula is equivalent to a quantifier-free (pp) formula (see \ref{elimq}).  There is also a dual form of this - every pp formula is equivalent to a divisibility formula (\ref{ppisdiv}).  We will also see in Section \ref{secelimimag} that the theory of ${\cal T}$ has elimination of pp-imaginaries - every pp-pair is definably isomorphic to a (quantifier-free) formula.

As with any theory of modules, the initial category of sorts, in this case a small version of $({\cal T}^{\rm c})^{\rm op}$, may be completed to the full category ${\mathbb L}({\cal T})^{\rm eq+}$ of pp-definable sorts:  the objects are pp-pairs and the morphisms are the pp-definable maps between these pairs (see Section \ref{secppsorts}).  In our context, this completed category of sorts has two manifestations.  One is the category of coherent functors \cite{KraCoh} on ${\cal T}$.  The other is a certain localisation of the category $({\rm mod}\mbox{-}{\cal T}^{\rm c},{\bf Ab})^{\rm fp}$ of finitely presented functors from ${\rm mod}\mbox{-}{\cal T}^{\rm c}$ - the category of finitely presented right ${\cal T}^{\rm c}$-modules - to the category ${\bf Ab}$ of abelian groups.  In fact, \cite[7.1/7.2]{PreAxtFlat}, this localisation turns out to be equivalent to the opposite of ${\rm mod}\mbox{-}{\cal T}^{\rm c}$ which is, in turn, equivalent to ${\cal T}^{\rm c}\mbox{-}{\rm mod}$.  The latter equivalence, \ref{rtlmods}, reflects the fact that the absolutely pure = fp-injective ${\cal T}^{\rm c}$-modules coincide with the flat ${\cal T}^{\rm c}$-modules.  We will, in Section \ref{secppsorts}, give details of this, as well as the action of each of these manifestations of ${\mathbb L}({\cal T})^{\rm eq+}$ on ${\cal T}$, respectively on $y{\cal T}$.

Free realisations and pp-types are used a lot in the model theory of modules and applications, so in Section \ref{secfreereal} we point out how these look in ${\cal T}$.

In Section \ref{secdefsub} we present the various types of data which can specify a definable subcategory of ${\cal T}$.  In Section \ref{sectorsTc} we see the bijection between definable subcategories of ${\cal T}$ and hereditary torsion theories of finite type on ${\rm Mod}\mbox{-}{\cal T}^{\rm c}$ and in Section \ref{secdefabs} we explore that connection in more detail.  The category of imaginaries of a definable subcategory is described in Section \ref{secmoddef}.  Some connections between hom-orthogonal pairs in ${\cal T}$ and hereditary torsion theories on ${\rm Mod}\mbox{-}{\cal T}^{\rm c}$ are seen in Section \ref{sectors} and this is continued in Section \ref{sectriangdef} with the bijection between triangulated definable subcategories and smashing subcategories of ${\cal T}$.

Section \ref{secspec} describes spectra associated to ${\cal T}$ and this is continued for tensor-triangulated categories in Section \ref{secttspec}.

For definable subcategories of module categories there is a duality, elementary duality, which exists at a number of levels, in particular between definable subcategories of ${\rm Mod}\mbox{-}R$ and $R\mbox{-}{\rm Mod}$.  This carries over, at least to algebraic triangulated categories; we outline that in Section \ref{secelemdual}.  If ${\cal T}$ is tensor-triangulated with ${\cal T}^{\rm c}$ rigid, then there is also an {\em internal} duality, induced by the duality on ${\cal T}^{\rm c}$; that is described in Section \ref{secintdual}.

Tensor-closed definable subcategories are briefly considered in Section \ref{secttcats}, and in Section \ref{secinthom} there is some exploration of the wider possibilities for interpreting the model-theoretic language.

\vspace{4pt}

Background on the model theory of modules can be found in various references; we use \cite{PreNBK} as a convenient compendium of results and references to the original papers.  We give a few reminders in this paper.  The approach in \cite{PreNBK} is algebraic/functor-category-theoretic; readers coming from model theory might find \cite{PreBk} or \cite{PreMTM} a more approachable introduction.  For model theory of modules over many-sorted rings, see \cite{PreErice}.

\vspace{4pt}

Thanks to Isaac Bird and Jordan Williamson for a number of useful comments and for sharing their preprint \cite{BirWilDualPr}.

\subsection{The restricted Yoneda functor}\label{secyon} \marginpar{secyon}

The restricted Yoneda functor $y:{\cal T} \to {\rm Mod}\mbox{-}{\cal T}^{\rm c}$, $X \to (-,X)\upharpoonright {\cal T}^{\rm c}$ underlies most of what we do here.  Restricting its domain to the category ${\cal T}^{\rm c}$ of compact objects gives, by the Yoneda Lemma and because ${\cal T}$ is idempotent-complete (see \cite[1.6.8]{NeeBk}), an equivalence between ${\cal T}^{\rm c}$ and the category ${\rm proj}\mbox{-}{\cal T}^{\rm c}$ of finitely generated projective right ${\cal T}^{\rm c}$-modules.  The functor $y$ is, however, neither full nor faithful and one effect of this is that the image of ${\cal T}$ in ${\rm Mod}\mbox{-}{\cal T}^{\rm c}$ is not closed under elementary equivalence, indeed it is not a definable subcategory (see Section \ref{secdefsub}) of ${\rm Mod}\mbox{-}{\cal T}^{\rm c}$.  We do, however, have \ref{ybijpi} and \ref{ybijcpct} below (the second is just by the Yoneda Lemma). 

First we recall (see \cite[\S 2.1.1]{PreNBK}) that an embedding $M \to N$ of objects in a module category, more generally in a definable additive category, is {\bf pure} if, for every pp formula $\phi$, the (image of the) solution set $\phi(M)$ is the intersection of $\phi(N)$ ``with $M$", meaning with the product of sorts of $M$ corresponding to the free variables of $\phi$.  And $M$ is {\bf pure-injective} if every pure embedding with domain $M$ is split.  There are many equivalent definitions, see \cite[\S\S 4.3.1, 4.3.2]{PreNBK}.

The theory of purity - intimately connected with solution sets of pp formulas and so with the model theory of additive structures - was developed, in algebraic terms, in compactly generated triangulated categories in \cite{BeligHomTriang} and \cite{KraTel}.  Essentially, it is the theory of purity in ${\rm Mod}\mbox{-}{\cal T}^{\rm c}$, more precisely, in the definable subcategory generated by $y{\cal T}$, pulled back to ${\cal T}$.  For example, $X\in {\cal T}$ is pure-injective iff $yX$ is a pure-injective ${\cal T}^{\rm c}$-module.  Since $yX$ is absolutely pure (\cite[Lemma 1.6]{KraTel}) that is equivalent to it being an injective ${\cal T}^{\rm c}$-module.  The pure-injective objects of ${\cal T}$ play the same key role that they do in the model theory of modules.  For instance every ($\emptyset$-)saturated module is pure-injective and the pure-injective modules are exactly the direct summands of saturated modules (see \cite[21.1/21.2]{PreMAMS} or \cite[2.9]{PreBk}); this is equally true in compactly generated triangulated categories.\footnote{This comment, like a few others, is particularly directed to those coming from model theory.}  

\begin{prop}\label{ybijpi} \marginpar{ybijpi} \cite[1.8]{KraTel} If $X\in {\cal T}$ is pure-injective then, for every $Y\in {\cal T}$, the restricted Yoneda map $y: (Y,X) \to (yY, yX)$ is bijective.
\end{prop}

\begin{prop}\label{ybijcpct} \marginpar{ybijcpct} If $A\in {\cal T}$ is a compact object then, for every $X\in {\cal T}$, the restricted Yoneda map $y: (A,X) \to (yA, yX)$ is bijective.
\end{prop}

In fact there is symmetry here in that \ref{ybijcpct} holds more generally for $A$ pure-projective (that is, a direct summand of a direct sum of compact objects).

We will use the fact that the restricted Yoneda functor induces an equivalence between the category ${\rm Pinj}({\cal T})$ of pure-injective objects in ${\cal T}$ and the category ${\rm Inj}\mbox{-}{\cal T}^{\rm c}$ of injective right ${\cal T}^{\rm c}$-modules.

\begin{theorem}\label{pinjtoinj} \marginpar{pinjtoinj} (\cite[1.9]{KraTel})  The restricted Yoneda functor $y:{\cal T} \to {\rm Mod}\mbox{-}{\cal T}^{\rm c}$ induces an equivalence 
$${\rm Pinj}({\cal T}) \simeq {\rm Inj}\mbox{-}{\cal T}^{\rm c}.$$
\end{theorem}

\subsection{Definable subcategories of module categories}\label{secdefmod} \marginpar{secdefmod}

Very briefly, we recall the context of the model theory of modules and the principal associated structures.  Some of this is defined more carefully later in the paper but see the references for more detail.

In model theory in general, the context is typically the category of models of some complete theory, with elementary embeddings.  In the context of modules, it turns out to be more natural to work with {\bf definable subcategories}, meaning full subcategories of module categories which are closed under elementary equivalence and which are {\bf additive}, meaning closed under direct sums and direct summands.  These subcategories are equivalently characterised, without reference to model theory, as follows (see \cite[\S 3.4]{PreNBK} for this and various other characterisations by closure conditions).

\begin{theorem}\label{chardefmod} \marginpar{chardefmod} A subcategory ${\cal D}$ of a module category is a definable subcategory iff ${\cal D}$ is closed under direct products, directed colimits and pure submodules.
\end{theorem}

If ${\cal X}$ is a set of modules, then we denote by $\langle {\cal X} \rangle$ the {\bf definable subcategory generated by} ${\cal X}$.  It is the closure of ${\cal X}$ under the above operations, equally it is the smallest additive subcategory containing ${\cal X}$ and closed under elementary equivalence.

It is the case, see \cite[3.4.8]{PreNBK}, that every definable subcategory is closed under pure-injective hulls where, if $M$ is a module, its {\bf pure-injective hull} $H(M)$, a minimal pure, pure-injective extension of $M$.\footnote{In fact, $M$ is an elementary submodule of $H(M)$, \cite[Cor.~4 to Thm.~4]{Sab2}.}  It follows that every definable subcategory is determined by the pure-injective modules in it.  If ${\cal T}$ is a compactly generated triangulated category and $X\in {\cal T}$, then the {\bf pure-injective hull} of $X$ may be defined to be the (unique-to-isomorphism over $X$, by \ref{pinjtoinj}) object $H(X)$ of ${\cal T}$ such that $yH(X) =E(yX)$, where $E$ denotes injective hull in the module category ${\rm Mod}\mbox{-}{\cal T}^{\rm c}$.

To each {\bf definable category} ${\cal D}$ - meaning a category equivalent to a definable subcategory of a module category - there is associated a skeletally small abelian category, ${\rm fun}({\cal D})$, of functors on ${\cal D}$.  This can be defined as the category of pp-imaginaries (see Section \ref{secppsorts}) for ${\cal D}$, or as a localisation of the free abelian category on $R$ where ${\cal D}$ is a definable subcategory of ${\rm Mod}\mbox{-}R$ ($R$ a possibly many-sorted ring), or as the category of {\bf coherent functors} - those that commute with direct products and directed colimits - from ${\cal D}$ to ${\bf Ab}$.  Each definable subcategory\footnote{The containing module category in \ref{chardefmod} may be replaced by any definable category.} ${\cal C}$ of ${\cal D}$ is determined by the Serre subcategory ${\cal S}_{\cal C}$ of ${\rm fun}({\cal D})$ which consists of those functors which are 0 on ${\cal C}$, and then ${\rm fun}({\cal C})$ is the (abelian) quotient category ${\rm fun}({\cal D})/{\cal S}_{\cal C}$ - the Serre localisation (see \cite[p.~30ff.]{KraBk}) of ${\rm fun}({\cal D})$ at ${\cal S}_{\cal C}$.

Also associated to a definable category ${\cal D}$ is its {\bf Ziegler spectrum} ${\rm Zg}({\cal D})$ (\cite{Zie}, see \cite[Chpt.~5]{PreNBK}) - a topological space whose points are the isomorphism classes of indecomposable pure-injective objects in ${\cal D}$ and whose open subsets are the complements of zero-sets of sets of coherent functors on ${\cal D}$.  The closed subsets of ${\rm Zg}({\cal D})$ are in natural bijection with the definable subcategories of ${\cal D}$, see \cite[5.1.6]{PreNBK}.  See Section \ref{secspec} for more on this.

\section{Model theory in compactly generated triangulated categories}\label{secmodth} \marginpar{secmodth}

We use formulas to specify the definable subsets of objects of ${\cal T}$.  In order to set these up, we choose a sub{\em set} ${\cal G}$ of ${\cal T}^{\rm c}$ which we will assume to be generating in the sense that, if $X\in {\cal T}$, then $(G,X)=0$ for every $G\in {\cal G}$ implies $X=0$, and we take the (opposite of the) full subcategory on ${\cal G}$ to be the category of sorts.  For convenience, we will assume that ${\cal G}$ is equivalent to ${\cal T}^{\rm c}$, that is, contains at least one isomorphic copy of each compact object of ${\cal T}$.  By ${\cal L}_{\cal G}$ we denote the resulting language, meaning the resulting set of formulas.  

We could take a smaller category of sorts, for instance, if ${\cal T}$ is monogenic, generated by a single compact object $S$, then we could consider the 1-sorted language based on $S$.  The obvious question is whether this would suffice, in the sense that every set definable in the larger language would also be definable in the 1-sorted language.  We don't pursue this here, but the relative approach and results in \cite{GarkPre1}, \cite{GarkPre2} should be helpful in answering this question.

In the other direction, we could make the maximal choice of sorts and use a language with the category ${\mathbb L}({\cal T})^{\rm eq+}$ of pp-imaginaries (see Section \ref{secppsorts}) for the sorts.  Since pp-imaginaries are already definable, this does not increase the collection of definable subsets.  For most purposes the choice of category of sorts does not matter provided the definable subsets are the same.  However, elimination of quantifiers and elimination of imaginaries are language-dependent, rather than structure-dependent.  Our choice of ${\cal G}$ as (essentially) ${\cal T}^{\rm c}$ is exactly analogous to basing a language for the model theory of $R$-modules ($R$ a 1-sorted ring) on the category ${\rm mod}\mbox{-}R$ of finitely presented modules, rather than using the 1-sorted language based on the single module $R_R$ (see \cite{PreMakVol} for more on choices of languages for additive categories).

\vspace{4pt}

Having chosen ${\cal G}$ we introduce a sort $s_A$ for each $A\in {\cal G}$ and a symbol for addition (and a symbol for the zero) on each sort and, for each $f:A\to B$ in ${\cal G}$, a corresponding function symbol from sort $B$ to sort $A$ to represent multiplication by $f$ = composition with $f$.  Note that the morphisms of ${\cal G}$ are the ``elements of the ring-with-many-objects ${\cal G}$".

Each object $X\in {\cal T}$ then becomes a structure for this language by taking its elements of sort $s_A$ to be the elements of $(A,X)$ and then interpreting the function symbols in the usual/obvious way.

\begin{remark}\label{rmkinthom} \marginpar{rmkinthom}  If ${\cal T}$ is tensor-triangulated and has an internal hom functor right adjoint to $\otimes$, then these sorts, which by definition are abelian groups, can be taken instead to be objects of ${\cal T}$, in the sense that we could interpret the sort $s_A(X)$ to be the internal hom object $[A,X]\in {\cal T}$.  In this ``internal" interpretation of the language, we have, since $(A,X) \simeq (\mathbbm{1}, [A,X])$ where $\mathbbm{1}$ is the tensor-unit, the (usual) elements of $X$ of sort $A$ identified with the morphisms $\mathbbm{1} \to [A,X]$.
\end{remark}

We will write ${\cal L}({\cal T})$, or just ${\cal L}$ for the language.  Since we assume that ${\cal G}$ is equivalent to ${\cal T}^{\rm c}$, the ${\cal L}({\cal T})$-structure $X\in {\cal T}$, which is literally a right ${\cal G}$-module, may be identified with the image, $yX=(-,X)\upharpoonright {\cal T}^{\rm c}$, of $X$ under the restricted Yoneda functor $y:{\cal T} \to {\rm Mod}\mbox{-}{\cal T}^{\rm c}$.  Therefore the model theory of $X$ as an object of ${\cal T}$ is exactly that of $yX$ as a right ${\cal T}^{\rm c}$-module.  Indeed, ${\cal L}({\cal T})$ is equally a language for ${\cal T}$ and for the module category ${\rm Mod}\mbox{-}{\cal T}^{\rm c}$, but bear in mind that there are more ${\cal T}^{\rm c}$-modules than those which are in the image of ${\cal T}$ in ${\rm Mod}\mbox{-}{\cal T}^{\rm c}$, more even than in the definable subcategory of ${\rm Mod}\mbox{-}{\cal T}^{\rm c}$ which is generated by that image.

Indeed, the definable subcategory, $\langle y{\cal T} \rangle$, of ${\rm Mod}\mbox{-}{\cal T}^{\rm c}$ generated by the image of ${\cal T}$ is exactly the subcategory, ${\rm Flat}\mbox{-}{\cal T}^{\rm c} = {\rm Abs}\mbox{-}{\cal T}^{\rm c}$, consisting of the flat = absolutely pure\footnote{In `most' module categories the flat and absolutely pure modules have little overlap; the fact that they are equal over the ring ${\cal T}^{\rm c}$ is a very characteristic feature here.} ${\cal T}^{\rm c}$-modules.

\begin{theorem} \label{abseqflat} \marginpar{abseqflat} \cite[8.11, 8.12]{BeligFreyd} \cite[2.7]{KraTel} If ${\cal T}$ is a compactly generated triangulated category and $y:{\cal T} \to {\rm Mod}\mbox{-}{\cal T}^{\rm c}$ is the restricted Yoneda functor, then $\langle y{\cal T}\rangle = {\rm Abs}\mbox{-}{\cal T}^{\rm c} ={\rm Flat}\mbox{-}{\cal T}^{\rm c}$
\end{theorem}

Therefore the model theory of ${\cal T}$ is the same as the model theory of the flat = absolutely pure right ${\cal T}^{\rm c}$-modules\footnote{${\cal T}^{\rm c}$ is both right and left coherent as a ring with many objects (see \cite[\S 4]{ObRoh}), which is why the flat and the absolutely pure objects form definable subcategories (see \cite[3.4.24]{PreNBK}).}.  The one difference is that some structures are missing from ${\cal T}$:  except in the case that ${\cal T}$ is pure semisimple \cite[9.3]{BeligHomTriang}, there are structures in $\langle y{\cal T}\rangle$ which are not in $y{\cal T}$.  However, the equivalence, \ref{pinjtoinj}, of categories ${\rm Pinj}({\cal T}) \simeq {\rm Inj}\mbox{-}{\cal T}^{\rm c}$ between the pure-injective objects of ${\cal T}$ and the injective ${\cal T}^{\rm c}$-modules, implies that $y{\cal T}$ does contain all the pure-injective models, in particular all the saturated models, of its theory.  It follows from \ref{abseqflat} that implications and equivalences of pp-formulas on ${\cal T}$ and on ${\rm Flat}\mbox{-}{\cal T}^{\rm c} = {\rm Abs}\mbox{-}{\cal T}^{\rm c}$ are the same.

\vspace{4pt}

For convenience we will sometimes write $(-,X)$ instead of $(-,X)\upharpoonright {\cal T}^{\rm c}= yX$ when $X\in {\cal T}$.

\subsection{The category of pp-sorts}\label{secppsorts} \marginpar{secppsorts}

Let $R$ be a, possibly multisorted, ring and let ${\cal D}$ be a definable subcategory of ${\rm Mod}\mbox{-}R$.  We recall how to define the {\bf category} ${\mathbb L}({\cal D})^{\rm eq+}$ {\bf of pp sorts} (or {\bf pp-imaginaries}) for ${\cal D}$.

First, for ${\cal D} ={\rm Mod}\mbox{-}R$, the category ${\mathbb L}({\rm Mod}\mbox{-}R)^{\rm eq+}$, more briefly denoted ${\mathbb L}_R^{\rm eq+}$, has, for its objects, the {\bf pp-pairs} $\phi/\psi$, that is pairs $(\phi,\psi)$ of pp formulas for $R$-modules with $\phi \geq \psi$, meaning $\phi(M) \geq \psi(M)$ for all $M\in {\rm Mod}\mbox{-}R$.  For its arrows, we take the pp-definable maps between these pairs.  See \cite[\S 1]{HerzDual} or \cite[\S 3.2.2]{PreNBK} for details and the fact that this category is abelian.  Each such pp-pair defines a coherent functor $M \mapsto \phi(M)/\psi(M)$ from ${\rm Mod}\mbox{-}R$ to ${\bf Ab}$ and every coherent functor has this form, see, for instance, \cite[\S 10.2]{PreNBK}.

For general ${\cal D}$, a definable subcategory of ${\rm Mod}\mbox{-}R$, we let $\Phi_{\cal D}$ be the Serre subcategory of ${\mathbb L}_R^{\rm eq+}$ consisting of those pp-pairs $\phi/\psi$ which are {\bf closed on}, that is $0$ on, every $M\in {\cal D}$ (that is, $\phi(M)=\psi(M)$ for every $M\in {\cal D}$).  Then ${\mathbb L}({\cal D})^{\rm eq+}$ is defined to be the quotient = Serre-localisation ${\mathbb L}_R^{\rm eq+}/\Phi_{\cal D}$.  So ${\mathbb L}({\cal D})^{\rm eq+}$ has the same objects as ${\mathbb L}_R^{\rm eq+}$ - the pp-pairs - and the morphisms in ${\mathbb L}({\cal D})^{\rm eq+}$ are given by pp formulas which on every $M\in {\cal D}$ define a function.  In particular the pp-pairs closed on ${\cal D}$ are isomorphic to $0$ in ${\mathbb L}({\cal D})^{\rm eq+}$.  The localised category ${\mathbb L}({\cal D})^{\rm eq+}$ also is abelian; in fact, see \cite[2.3]{PreRajShv}, every skeletally small abelian category arises in this way.

An equivalent \cite[12.10]{PreMAMS}, but less explicit, definition is that ${\mathbb L}({\cal D})^{\rm eq+} = ({\cal D}, {\bf Ab})^{\prod \rightarrow}$ - the category of functors\footnote{additive, as always assumed in this paper} from ${\cal D}$ to ${\bf Ab}$ which commute with direct products and directed colimits (that is, coherent functors, equivalently \cite[25.3]{PreMAMS} interpretation functors in the model-theoretic sense).

It is well-known, see \cite[10.2.37, 10.2.30]{PreNBK}, and much-used, that, for ${\cal D} = {\rm Mod}\mbox{-}R$, the category of pp-pairs is equivalent to the free abelian category on $R^{\rm op}$ and, also, that it can be realised as the category $({\rm mod}\mbox{-}R, {\bf Ab})^{\rm fp}$ of finitely presented functors on finitely presented modules (see \cite[10.2.30, 10.2.37]{PreNBK}) equivalently, as just said, it is equivalent to the category of coherent functors on all modules (see \cite[\S 10.2.8]{PreNBK}).  Then, for a general definable subcategory ${\cal D}$ of ${\rm Mod}\mbox{-}R$, we obtain ${\mathbb L}({\cal D})^{\rm eq+}$ as the Serre-quotient $({\rm mod}\mbox{-}R, {\bf Ab})^{\rm fp}/{\cal S}_{\cal D}$ where ${\cal S}_{\cal D}$ is the Serre subcategory of those functors $F\in ({\rm mod}\mbox{-}R, {\bf Ab})^{\rm fp}$ with $\overrightarrow{F}{\cal D}=0$.  Here $\overrightarrow{F}$ is the unique extension of (a finitely presented) $F:{\rm mod}\mbox{-}R \to {\bf Ab}$ to a (coherent) functor from ${\rm Mod}\mbox{-}R$ to ${\bf Ab}$ which commutes with directed colimits.  Often we simplify notation by retaining the notation $F$ for this extension $\overrightarrow{F}$.

Under the identification of ${\mathbb L}_R^{\rm eq+}$ and $({\rm mod}\mbox{-}R, {\bf Ab})^{\rm fp}$ the Serre subcategory $\Phi_{\cal D}$ is identified with ${\cal S}_{\cal D}$.

In applying this in our context, we use the following result, where ${\rm Flat}\mbox{-}R$ denotes the category of flat right $R$-modules and ${\rm Abs}\mbox{-}R$ denotes the category of absolutely pure = fp-injective right $R$-modules.  For the notion of a left coherent multisorted ring - one whose category of left modules is locally coherent - see \cite[4.1]{ObRoh}.

\begin{theorem} \label{funcatabsflat} \marginpar{funcatabsflat} \cite[7.1/7.2]{PreAxtFlat}  If $ R $ is any left coherent (multisorted) ring, then ${\rm Flat}\mbox{-}R$ is a definable subcategory of ${\rm Mod}\mbox{-}R$ and 
$${\mathbb L}({\rm Flat}\mbox{-}R)^{\rm eq+} \simeq R\mbox{-}{\rm mod}.$$ 
If $R$ is a right coherent ring, then ${\rm Abs}\mbox{-}R$ is a definable subcategory of ${\rm Mod}\mbox{-}R$ and
$${\mathbb L}({\rm Abs}\mbox{-}R)^{\rm eq+} \simeq ({\rm mod}\mbox{-}R)^{\rm op}.$$
\end{theorem}

Because ${\cal T}^{\rm c}$ is right and left coherent, \cite[8.11, 8.12]{BeligFreyd}, and since ${\rm Abs}\mbox{-}{\cal T}^{\rm c} = {\rm Flat}\mbox{-}{\cal T}^{\rm c}$, we have the following corollary.

\begin{cor} \label{rtlmods} \marginpar{rtlmods}  If ${\cal T}$ is a compactly generated triangulated category, then there is an equivalence
$$d:{\cal T}^{\rm c}\mbox{-}{\rm mod} \simeq ({\rm mod}\mbox{-}{\cal T}^{\rm c})^{\rm op}$$
and this category is equivalent to the category ${\mathbb L}({\cal T})^{\rm eq+}$ of pp-imaginaries for ${\cal T}$.
\end{cor}

We write $d$ for the (anti-)equivalence in each direction.

There is another description of the category appearing in \ref{rtlmods}.  We say that a {\bf coherent} functor on ${\cal T}$ is one which is the cokernel of a map between representable functors $(A,-):{\cal T} \to {\bf Ab}$ with $A\in {\cal T}^{\rm c}$.  Explicitly, if $f:A\to B$ is in ${\cal T}^{\rm c}$ then we obtain an exact sequence of functors on ${\cal T}$:
$$(B,-) \xrightarrow{(f,-)} (A,-) \to F_f \to 0;$$
and the cokernel $F_f$ is a typical coherent functor on ${\cal T}$. 

In module categories having a presentation of this form, with $A$ and $B$ finitely presented, is equivalent to commuting with products and directed colimits but triangulated categories don't have directed colimits.  There is the following analogous characterisation.

\begin{theorem}\label{cohtriang} \marginpar{cohtriang} \cite[5.1]{KraCoh}  Suppose that ${\cal T}$ is a compactly generated triangulated category.  Then $F:{\cal T} \to {\bf Ab}$ is a coherent functor iff $F$ commutes with products and sends homology colimits to colimits.
\end{theorem} 

We denote the category of coherent functors on ${\cal T}$, with the natural transformations between them, by ${\rm Coh}({\cal T})$.  This category is abelian; in fact we have the following.

\begin{theorem} \label{cohfpdual} \marginpar{cohfpdual} \cite[7.2]{KraCoh} There is a duality 
$$({\rm mod}\mbox{-}{\cal T}^{\rm c})^{\rm op} \simeq {\rm Coh}({\cal T})$$ and hence $${\rm Coh}({\cal T}) \simeq \, {\cal T}^{\rm c}\mbox{-}{\rm mod}.$$
\end{theorem}

Indeed, to go from ${\rm Coh}({\cal T})$ to ${\cal T}^{\rm c}\mbox{-}{\rm mod}$ we just restrict the action of $F\in {\rm Coh}({\cal T})$ to ${\cal T}^{\rm c}$ and, in the other direction, we apply the projective presentation $(B,-) \to (A,-) \to G \to 0$ of a finitely presented left ${\cal T}^{\rm c}$-module in ${\cal T}$ and we get a coherent functor.  The category ${\mathbb L}({\cal T})^{\rm eq+} $ of pp-definable sorts and pp-definable maps for ${\cal T}$ is defined just as for a module category.  Since the model theories of ${\cal T}$ and ${\rm Flat}\mbox{-}{\cal T}^{\rm c}$ are identical, it is exactly ${\mathbb L}({\rm Flat}\mbox{-}{\cal T}^{\rm c})^{\rm eq+}$.

\begin{cor}\label{ppcat} \marginpar{ppcat}  The category of pp-imaginaries for a compactly generated triangulated category ${\cal T}$ can be realised in the following forms
$${\mathbb L}({\cal T})^{\rm eq+} \simeq \, {\rm Coh}({\cal T})\,  \simeq \, {\cal T}^{\rm c}\mbox{-}{\rm mod}.$$
\end{cor}

The duality in \ref{cohfpdual} respects the actions of those categories of functors on ${\cal T}$.  We give the details.

The action of ${\rm Coh}({\cal T})$ on ${\cal T}$ is given by the exact sequence above presenting $F_f$:  if $X\in {\cal T}$, then $F_f(X)$ is defined by exactness of the sequence
$$(B,X) \to (A,X) \to F_f(X) \to 0.$$  

The action of ${\rm mod}\mbox{-}{\cal T}^{\rm c}$ on ${\cal T}$ is given by Hom applied after the restricted Yoneda functor $y$.  Explicitly:  the typical finitely presented right ${\cal T}^{\rm c}$-module $G_f$ is given by an exact sequence (a projective presentation)
$$yA \xrightarrow{yf} yB \to G_f \to 0,$$
that is
$$(-,A) \xrightarrow{(-,f)} (-,B) \to G_f \to 0,$$
where $A\xrightarrow{f}B \in {\cal T}^{\rm c}$.
The action of $G_f$ on $X\in {\cal T}$ is induced by the action of $(-,yX)$ on it:  we have the exact sequence
$$0 \to (G_f,(-,X)) \to ((-,B),(-,X)) \xrightarrow{((-,f),(-,X))} ((-,A),(-,X)),$$
that is
$$0 \to G_f(X) \to (B,X) \xrightarrow{(f,X)} (A,X),$$
defining the value of $G_f$ on the typical object $X\in {\cal T}$.  So, if $G\in {\rm mod}\mbox{-}{\cal T}^{\rm c}$ and $X\in {\cal T}$, then the action of $G$ on $X$ is defined by
$$G(X) = (G,yX).$$

Notice that the morphism $f:A\to B$ in ${\cal T}^{\rm c}$ has given rise to the exact sequence of abelian groups:
\begin{equation}
\label{eq1} 0 \to G_f(X) \to (B,X) \xrightarrow{(f,X)} (A,X) \to F_f(X) \to 0.
\end{equation}

The duality $({\rm mod}\mbox{-}{\cal T}^{\rm c})^{\rm op} \simeq {\rm Coh}({\cal T})$ in \ref{cohfpdual} takes a finitely presented right ${\cal T}^{\rm c}$-module $G$ to the coherent functor
$$G^\circ:X\mapsto (G,yX) = (G,(-,X))$$
for $X\in {\cal T}$ - that is, the action we defined just above.  
This takes the representable functor $G=(-,A)$ where $A\in {\cal T}^{\rm c}$, to the representable coherent functor $(A,-):{\cal T} \to {\bf Ab}$.  Therefore, the 4-term exact sequence (\ref{eq1}) above can be read as the application of the following exact sequence of functors in ${\rm Coh}({\cal T})$ to $X$.
\begin{equation} \label{eq2}
0 \to G_f^\circ \to (B,-) \xrightarrow{(f,-)} (A,-) \to F_f \to 0.
\end{equation}

In the other direction, the duality $({\rm mod}\mbox{-}{\cal T}^{\rm c})^{\rm op} \simeq {\rm Coh}({\cal T})$ takes $F\in {\rm Coh}({\cal T})$ to the finitely presented ${\cal T}^{\rm c}$-module
$$F^\diamond: C \mapsto (F, (C,-))$$
for $C\in {\cal T}^{\rm c}$.  So $(A,-)^\diamond = (-,A)$.  If $F=F_f$, then applying $(-, (C,-))$ to the presentation (\ref{eq2}) of $F_f$ and using that $(C,-)$ is injective in ${\rm Coh}({\cal T})$ (by \ref{cohfpdual} and since $(-,C)$ is projective in ${\rm Mod}\mbox{-}{\cal T}^{\rm c}$) allows us to read the resulting 4-term exact sequence as the application, to $C \in {\cal T}^{\rm c}$, of the following exact sequence of functors in ${\rm mod}\mbox{-}{\cal T}^{\rm c}$.
\begin{equation}\label{eq3}
0 \to F_f^\diamond \to (-,A) \xrightarrow{(-,f)} (-,B) \to G_f \to 0
\end{equation}

Applying the duality-equivalences $(-)^\diamond:({\rm Coh}({\cal T}))^{\rm op} \to {\rm Mod}\mbox{-}{\cal T}^{\rm c}$ and $(-)^\circ:({\rm Mod}\mbox{-}{\cal T}^{\rm c})^{\rm op} \to {\rm Coh}({\cal T})$ interchanges (\ref{eq2}) - an exact sequence in ${\rm Coh}({\cal T})$ - and (\ref{eq3}) - an exact sequence in ${\rm mod}\mbox{-}{\cal T}^{\rm c}$.

\vspace{4pt}

The equivalences of these functor categories with the category ${\mathbb L}({\cal T})^{\rm eq+}$ of pp-pairs for ${\cal T}$ are given explicitly on objects as follows.  Let $f:A \to B$ be a morphism in ${\cal T}^{\rm c}$, so $F_f$ is a typical coherent functor.  
$\xymatrix{A \ar[r]^f \ar[d]_{x_A} & B \ar@{.>}[dl]^{y_B} \\
X}$

We have that $F_fX = (A,X)/{\rm im}(f,X)$ and hence $F_f$ is the functor given by the pp-pair $(x_A=x_A)/(\exists y_B \,\, x_A=y_Bf)$, that is 
$$F_f = (x_A=x_A)/(f|x_A).$$

\noindent We use subscripts on variables to show their sorts but might sometimes drop them for readability.  We also use variables (which really belong in formulas) to label morphisms (for which they are place-holders) in what we hope is a usefully suggestive way.

Also, from the exact sequence (\ref{eq1}), we see that $G_f^\circ (-)= {\rm ker}(f,-)$ and so is the functor given by the pp-pair 
$$G_f^\circ = (x_Bf=0)/(x_B=0).$$

Since the duality ${\rm Coh}({\cal T}) \simeq ({\rm mod}\mbox{-}{\cal T}^{\rm c})^{\rm op}$ preserves the actions on ${\cal T}$, these pp-pairs also give the actions of, respectively, $F_f^\diamond$ and $G_f$ on ${\cal T}$.

To go from pp-pairs to functors, we may use \ref{ppprfmla} below, which says that every pp-pair is isomorphic to one of a form seen above, namely $xf=0/x=0$.

\subsection{Elimination of quantifiers}\label{secelimq} \marginpar{secelimq}

If a ring $R$ is right coherent then every pp formula is equivalent on ${\rm Abs}\mbox{-}R$ to an annihilator formula and, if $R$ is left coherent, then every pp formula on ${\rm Flat}\mbox{-}R$ is equivalent to a divisibility formula (see \cite[2.3.20, 2.3.9+2.3.19]{PreNBK}).  These results are equally valid for rings with many objects (because any formula involves only finitely many sorts, so is equivalent to a formula over a ring with one object).  It follows that the theory of ${\cal T}$ has elimination of quantifiers, indeed it has the stronger property elim-q$^+$, meaning that each pp formula is equivalent to a quantifier-free pp formula, that is, to a conjunction of equations\footnote{Indeed, since our sorts are closed under finite direct sums, every pp formula is equivalent to a single equation}.  Also ${\cal T}$ has the elementary-dual elimination of pp formulas to divisibility formulas.  But it is instructive to see exactly how this works when the ring is the category ${\cal T}^{\rm c}$ of compact objects of a compactly generated triangulated category ${\cal T}$.  This is an expansion of \cite[3.1, 3.2]{GarkPre2}.  We write $0$ for any $n$-tuple $(0, \dots, 0)$.

\vspace{4pt}

Given $f:A\to B$ in ${\cal T}^{\rm c}$, form the distinguished triangle\footnote{We will often write ``triangle" meaning distinguished triangle.} as shown.
$$A\xrightarrow{f} B\xrightarrow{g} C \to \Sigma A$$  
Since ${\cal T}^{\rm c}$ is triangulated, $C\in {\cal T}^{\rm c}$.  Since representable functors on a triangulated category are exact (meaning that they take triangles to (long) exact sequences), for every $X\in {\cal T}$, $(C,X) \xrightarrow{(g,X)} (B,X) \xrightarrow{(f,X)} (A,X)$ is exact so, for $x_B\in (B,X)$, we have $x_B \in {\rm ker}(f,X)$ iff $x_B \in {\rm im}(g,X)$, that is, $x_Bf=0$ iff $g\mid x_B$ that is, iff $\exists y_C\, (x_B = y_Cg)$.  Thus 
$$x_Bf=0 \quad \Leftrightarrow \quad g\mid x_B.$$

Since ${\cal T}^{\rm c}$ has finite direct sums, tuples of variables may be wrapped up into single variables (we do this explicitly below), so these formulas are general annihilator and divisibility formulas.  Therefore every annihilator formula is equivalent to a divisibility formula and {\it vice versa}.  We record this.

\begin{prop}\label{anneqdiv} \marginpar{anneqdiv} If $A\xrightarrow{f} B\xrightarrow{g} C \to \Sigma A$ is a distinguished triangle, then the formula $x_Bf=0$ is equivalent to $g\mid x_B$.
\end{prop}

Before continuing, note that, because ${\cal T}^{\rm c}$ is closed under finite direct sums, a finite sequence $(x_1, \dots, x_n)$ of variables, with $x_i$ of sort $A_i$, may be regarded as a single variable of sort $A_1 \oplus \dots \oplus A_n$.  That simplifies notation and allows us to treat a general pp formula as one of the form $\exists x_{B'} \, (x_Bf=x_{B'}f')$, that is, $f'|xf$ for short.
$\xymatrix{& B \ar[drr]^{x_B} \\ A\ar[ur]^f \ar[dr]_{f'} & & & (-)\\& B' \ar[urr]_{x_{B'}}}$
That is equivalent to 
$\quad$ $\exists x_{B'} \, \Big((x_B,x_{B'})\Big( \begin{array}{c}f \\ f'\end{array} \Big) =0\Big)$ $\quad\quad$ $\xymatrix{A \ar[r]^{\left( \begin{smallmatrix}f \\ f' \end{smallmatrix} \right)} & B\oplus B' \ar[d]^{(x_{B}, x_{B'})} \\ & (-)}$.

\noindent So form the triangle $A\xrightarrow{\left( \begin{smallmatrix}f \\ f' \end{smallmatrix} \right)} B\oplus B' \xrightarrow{\overline{g} = (g,g')} C \to \Sigma A$.  

\noindent By \ref{anneqdiv} above, the formula $\exists\,  x_{B'} \, \big((x_B,x_{B'})\Big( \begin{array}{c}f \\ f'\end{array} \Big) =0\big)$ is equivalent to $\exists \, x_{B'} \, \exists x_C \, \big((x_{B}, x_{B'}) = x_C\overline{g}\big)$, that is to 
$$\exists x_{B'}\, \exists x_C \, (x_B=x_Cg\, \wedge \, x_{B'}=x_Cg'),$$ 
and the $x_{B'}$ is irrelevant now (set $x_{B'} = x_Cg'$).  So the original formula is equivalent to $g\mid x_B$ where $g$ is, up to sign, the map which appears in the weak pushout $\xymatrix{A \ar[r]^f \ar[d]_{f'} & B \ar[d]^g \\ B' \ar[r]_{g'} & C}$.  Let us record that.  

\begin{lemma}\label{ppisdiv} \marginpar{ppisdiv}  Given morphisms $f,f':A\to B$ in ${\cal T}^{\rm c}$, the (typical pp) formula 
$$\exists \, x_{B'} \, (x_Bf=x_{B'}f')$$ 
is equivalent to the divisibility formula $g|x_B$, where $g$ is as in the distinguished triangle
$$A\xrightarrow{\left( \begin{smallmatrix}f \\ f' \end{smallmatrix} \right)} B\oplus B' \xrightarrow{(g, g')} C \to \Sigma A,$$ 
and hence is also equivalent to the annihilation formula $x_Bf''=0$ where 
$$A' \xrightarrow{f''} B' \xrightarrow{g} C \to \Sigma A'$$ 
is a distinguished triangle.
\end{lemma}

Thus every pp formula is equivalent on ${\cal T}$ to a divisibility formula and hence also to an annihilator formula.  In particular:

\begin{theorem}\label{elimq} \marginpar{elimq} \cite[3.1, 3.2]{GarkPre2} If ${\cal T}$ is a compactly generated triangulated category and ${\cal L}$ is the language for ${\cal T}$ based on ${\cal T}^{\rm c}$, then (the theory of\footnote{Meaning that every completion of the theory of ${\cal T}$ has elimination of quantifiers and the elimination is uniform over these completions.}) ${\cal T}$ has elimination of quantifiers, indeed has elim-q$^+$.
\end{theorem}

\subsection{Types and free realisations}\label{secfreereal} \marginpar{secfreereal}

We start with a little model theory but soon come back to the algebra.

If $A_1, \dots, A_n$ are compact objects of ${\cal T}$ and if $a_i:A_i \to X\in {\cal T}$ are elements of $X\in {\cal T}$, then the {\bf type} of $\overline{a} = (a_1, \dots, a_n)$ (in $X$) is the set of formulas $
\chi$ such that $\overline{a}\in \chi(X)$.  The {\bf pp-type} of $\overline{a} \in X$ is 
$${\rm pp}^X(\overline{a}) = \{ \phi \text{ pp}: \overline{a} \in \phi(X)\}.$$

Since we have pp-elimination of quantifiers (\ref{ppeq}) the type of $\overline{a}$ in $X$ is determined by its subset ${\rm pp}^X(\overline{a})$.  Indeed it is equivalent, modulo the theory of ${\cal T}$ (equivalently, the theory of absolutely pure = flat ${\cal T}^{\rm c}$-modules) to the set ${\rm pp}^X(\overline{a}) \, \cup\, \{\neg \psi: \psi \text{ pp and } \psi \notin {\rm pp}^X(\overline{a})\}$.\footnote{This is also true for types with parameters but we don't use these in this paper.  For more on this see, for instance, \cite[2.20]{PreBk}.}

As remarked already, because ${\cal T}^{\rm c}$ has finite direct sums, we can replace a tuple $(x_1, \dots, x_n)$ of variables $x_i$ of sort $A_i$ by a single variable of sort $A_1 \oplus \dots \oplus A_n$ (and, similarly, tuples of elements may be replaced by single elements).  So any pp-definable subgroup of an object $X \in {\cal T}$ - that is, the solution set $\phi(X)$ in $X$ of some pp formla $\phi$ - can be taken to be a subgroup of $(A,X)$ for some $A \in {\cal T}^{\rm c}$.

We say that two formulas are {\bf equivalent} (on ${\cal T}$) if they have the same solution set in every $X\in {\cal T}$.  There is an ordering on the set of (equivalence classes of) pp formulas:  if $\phi$, $\psi$ are pp formulas in the same free variables, then we set $\phi \leq \psi$ iff $\forall X \in {\cal T}$, $\phi(X) \leq \psi(X)$.  This (having fixed the free variables) is a lattice with meet given by conjunction $\phi \wedge \psi$ (defining the intersection of the solution sets) and join given by sum $\phi + \psi$ (defining the sum of the solution sets).

By a {\bf pp-type} (without parameters) we mean a deductively closed set of pp formulas, equivalently a filter (i.e.~meet- and upwards-closed) in the lattice of (equivalence classes of) pp formulas (always with some fixed sequence of free variables).  We note the following analogue of the module category case (see \cite[1.2.23]{PreNBK}).

\begin{lemma}  Suppose that ${\cal T}$ is a compactly generated triangulated category and $\phi$, $\psi$ are pp formulas with the same free variables.  Then $\phi \leq \psi$ iff for all $A\in {\cal T}^{\rm c}$ we have $\phi(A) \leq \psi(A)$.
\end{lemma}
\begin{proof} Suppose that for all $A\in {\cal T}^{\rm c}$ we have $\phi(A) \leq \psi(A)$ and let $X\in {\cal T}$.  Since $yX$ is a flat object of ${\rm Mod}\mbox{-}{\cal T}^{\rm c}$, it is the direct limit of some directed diagram of finitely generated projective ${\cal T}^{\rm c}$-modules.  The latter all have the form $yA$ for some $A\in {\cal T}^{\rm c}$.  Since, for any pp formula $\phi$, $\phi(-)$ commutes with direct limits (see \cite[1.2.31]{PreNBK}), we conclude that $\phi(yX) \leq \psi(yX)$, and hence that $\phi(X) \leq \psi(X)$, as required.
\end{proof}

In the above proof we made the (harmless and useful) identification of pp formulas for objects of ${\cal T}$ and for right ${\cal T}^{\rm c}$-modules.

Suppose that $p$ is a pp-type, consisting of pp formulas with free variables $x_1, \dots, x_n$ where $x_i$ has sort (labelled by) $A_i \in {\cal T}^{\rm c}$.  Then, by \cite[3.3.6, 4.1.4]{PreNBK},  $p$ has a {\bf realisation} in some object $M$ in the definable subcategory $\langle y{\cal T} \rangle$ of ${\rm Mod}\mbox{-}{\cal T}^{\rm c}$, meaning there is a tuple $\overline{b}$ of elements in $M$ with ${\rm pp}^M(\overline{b}) =p$.  Pp-types are unchanged by pure embeddings and every such module $M$ is a pure, indeed elementary, subobject of its pure-injective (= injective) hull, which has the form $yX$ for some $X \in {\cal T}$.  So we obtain a realisation of $p$ in some object $X\in {\cal T}$:  there is $\overline{a} = (a_1, \dots, a_n)$ with $a_i:A_i \to X$ such that ${\rm pp}^X(\overline{a}) =p$.  The object $X$ is pure-injective in ${\cal T}$ (\ref{pinjtoinj}) and, moreover, may be chosen to be minimal such\footnote{Corresponding to the injective hull of the submodule of $M$ generated by the entries of $\overline{b}$.}, in which case it is denoted $H(p)$ - the {\bf hull} of $p$.  This is unique up to isomorphism in the sense that if $N$ is a pure-injective object of ${\cal T}$ and if $\overline{c}$ is a tuple from $N$ with ${\rm pp}^N(\overline{c}) =p$, then there is an embedding of $H(p)$ into $N$ as a direct summand, taking $\overline{a}$ to $\overline{c}$ and this will be an isomorphism if $N$ also is minimal over $\overline{c}$.  See \cite[\S 4.3.5]{PreNBK} for this and related results - these all apply to any compactly generated triangulated category ${\cal T}$ because its model theory is really just that of a definable subcategory of ${\rm Mod}\mbox{-}{\cal T}^{\rm c}$, and because all the pure-injective objects of that definable subcategory are images of objects of ${\cal T}$.

If $\phi$ is a pp formula, then we have the pp-type it generates:  
$$\langle \phi \rangle = \{ \psi: \phi \leq \psi\}.$$ 
We say that a pp-type is {\bf finitely generated} ({\bf by} $\phi$) if it has this form for some $\phi$.

If $\phi$ is a pp formula with free variable of sort $A$ (without loss of generality we may assume that there is just one free variable) then a {\bf free realisation} of $\phi$ is a pair $(C,c_A)$ where $C \in {\cal T}^{\rm c}$ and $c_A:A \to C$ is an element of $C$ of sort $A$ with ${\rm pp}^C(c_A) = \langle \phi \rangle$.  We have the following analogue to \cite[1.2.7]{PreNBK}.  In the statement of this result, we continue to overuse notation by allowing $x_A$ to denote an element of sort $A$ (in addition to our use of $x_A$ to denote a variable of sort $A$).

\begin{lemma} Suppose $\phi$ is a pp formula with free variable $x_A$ (for some $A\in {\cal T}^{\rm c})$.  Let $C\in {\cal T}^{\rm c}$ and suppose $c_A \in (A,C)$ with $c_A\in \phi(C)$.  Then $(C,c_A)$ is a free realisation of $\phi$ iff for every $x_A:A \to X\in {\cal T}$ such that $x_A \in \phi(X)$, there is a morphism $h:C \to A$ with $hc_A=x_A$.
\end{lemma}
\begin{proof}  Existence of free realisations in ${\cal T}$ (\ref{existfreereal} below) gives the direction ($\Leftarrow$) since, if $(B,b)$ is a free realisation of $\phi$, then there is a morphism $g:C \to B$ with $gc_A =b$, so ${\rm pp}^C(c_A) =\langle \phi\rangle$ (because morphisms are non-decreasing on pp-types - see \cite[1.2.8]{PreNBK}).  For the converse, if $a\in \phi(X)$, then $ya \in \phi(yX)$\footnote{For clarity, the language for ${\cal T}$ is exactly the language for ${\rm Mod}\mbox{-}{\cal T}^{\rm c}$ and the definition of the solution set  $\phi(X)$ is identical to the definition of the solution set of $\phi(yX)$.}  Since the pp-type of $yc_A$ in $yC$ is exactly that of $c_A$ in $C$, it is generated by $\phi$ and hence, since $ya\in \phi(yX)$, there is, by \cite[1.2.7]{PreNBK}, a morphism $f':yC \to yX$ with $f'\cdot yc_A =ya$.  Because $C\in {\cal T}^{\rm c}$, there is, by \ref{ybijcpct}, $f:C \to X$ with $f'=yf$.  Therefore $y(fc_A) = ya$ so, again by \ref{ybijcpct}, $fc_A =a$, as required.
\end{proof}

We show that every pp formula in the language for ${\cal T}$ has a free realisation in ${\cal T}$.  We use the fact that every formula is equivalent to a divisibility formula.

If a morphism $f$ {\bf factors initially through} a morphism $g$ - that is, $f=hg$ for some $h$ - then write $g \geq f$.

\begin{lemma} If $f:A \to B$ is a morphism in ${\cal T}^{\rm c}$ then the pp-type, $\langle f|x_A \rangle$, generated by the formula $f|x_A$ is, up to equivalence of pp formulas, $\{ g|x_A :  g \geq f \}$.
\end{lemma}
\begin{proof}  By \ref{ppisdiv} every pp formula is equivalent to a divisibility formula, so we need only consider formulas of the form $g|x_A$.

If $g \geq f$, say $g:A \to C$ and $f=hg$ with $h:C \to B$, then, for any $x_A:A \to X \in {\cal T}$ with $f|x_A$, say $x_A = x_Bf$, we have $x_A =x_Bhg = x_Cg$ with $x_C = x_Bg$, so we have $g|x_A$.  That is, $g|x_A \in \langle f|x_A \rangle$.

For the converse, if $g:A \to C$ is in ${\cal T}^{\rm c}$ and $g|x_A \in \langle f|x_A \rangle$, then, applying this with $X=B$ and $x_A =f$, we obtain that there is $h:C \to B$ such that $hg=f$, and $g \geq f$, as required.
\end{proof}

\begin{cor}\label{existfreereal} \marginpar{existfreereal}  Suppose that $\phi(x_A)$ is a pp formula for the language of ${\cal T}$.  Choose (by \ref{ppisdiv}) a morphism $f:A \to B$ in ${\cal T}^{\rm c}$ such that $\phi$ is equivalent to $f|x_A$.  Then $(B,f)$ is a free realisation of $\phi$.
\end{cor}

\subsection{Elimination of imaginaries}\label{secelimimag} \marginpar{secelimimag}

Next we prove elimination of pp-imaginaries:  we show that every pp-pair is isomorphic, in the category ${\mathbb L}({\cal T})^{\rm eq+}$ of pp-pairs, to a pp formula, indeed by \ref{elimq}, to a quantifier-free pp formula if we identify a pp formula $\phi(\overline{x})$ with the pp-pair $\phi(\overline{x})/(\overline{x}=0)$ in ${\mathbb L}({\cal T})^{\rm eq+}$.

Recall, \ref{ppcat}, that the category of pp-imaginaries is equivalent to the category ${\rm Coh}({\cal T})$ of coherent functors on ${\cal T}$.  So let us take a coherent functor $F_g$ defined by the exact sequence $(C,-) \xrightarrow{(g,-)} (B,-) \to F_g \to 0$ for some $g:B\to C$ in ${\cal T}^{\rm c}$.  We have the distinguished triangle $A \xrightarrow{f} B \xrightarrow{g} C \xrightarrow{h} \Sigma A$ and extend it to  $\Sigma^{-1}C \xrightarrow{\Sigma^{-1}h} A \xrightarrow{f} B \xrightarrow{g} C \xrightarrow{h} \Sigma A$ then consider the exact sequence of functors on ${\cal T}$:
$$(\Sigma A,-) \xrightarrow{(h,-)} (C,-) \xrightarrow{(g,-)} (B,-) \xrightarrow{(f,-)} (A,-) \xrightarrow{(\Sigma^{-1}h,-)} (\Sigma^{-1}C,-)$$ 
where we have the factorisation 
$\xymatrix{(B,-) \ar[rr]^{(f,-)} \ar@{->>}[dr] && (A,-) \\ 
& F_g \ar@{^{(}->}[ur]}$.  

So $F_g\simeq {\rm im}(f,-)$ in $(A,-)$ and therefore $F_g$ is isomorphic to the functor given by the pp formula $f\mid x_A$ which, by \ref{anneqdiv}, is equivalent to the quantifier-free pp formula $x_A\cdot \Sigma^{-1}h=0$; that is $F_g \simeq G_{\Sigma^{-1}h}^\circ$ (this is also clear from the above exact sequence).  Thus we have the following.

\begin{theorem}\label{ppprfmla} \marginpar{ppprfmla} \cite[4.3]{GarkPre2} Every pp-pair is pp-definably isomorphic to a pp formula which may be taken to be quantifier-free (alternatively a divisibility formula).  Thus, (the theory of) ${\cal T}$ has elimination of pp imaginaries.

Explicitly, if $g:B\to C$ is in ${\cal T}^{\rm c}$ then the (typical)  pp-pair $F_g= {\rm coker}((g,-):(C,-) \to (B,-))$ is equivalent to the divisibility formula $f|x_A$ and to the annihilation formula $x_A\Sigma^{-1}h=0$ where $f$ and $h$ are such that $\Sigma^{-1}C \xrightarrow{\Sigma^{-1}h}  A \xrightarrow{f} B \xrightarrow{g} C ( \xrightarrow{h} \Sigma A)$ is a distinguished triangle.
\end{theorem}

\subsection{Enhancements, Ultraproducts}\label{secenhanultra} \marginpar{secenhanultra}

Arguments using reduced products, in particular ultraproducts, are often used in model theory.  In many cases their use can be replaced by arguments involving realising types in elementary extensions but in some cases the more algebraic and `explicit' (modulo use of the axiom of choice\footnote{needed to extend a filter to a non-principal ultrafilter}) ultraproduct construction is better.  At first sight we can't use ultraproducts in compactly generated triangulated categories because, even though typically they have direct products, they almost never have all directed colimits (recall, e.g. \cite[\S 3.3.1]{PreNBK}, that an ultraproduct is a directed colimit of direct products of its component structures).  Homotopy colimits along a countably infinite directed set are available but that is not enough to form ultraproducts.

In \cite{LakDer} Laking introduced ultraproducts in this context by using Grothendieck derivators.  We don't go into the details here but see \cite[\S 2]{LakDer} for the construction of coherent reduced products for derivators.  In \cite{LakVit} a different approach, using dg-categories and model categories, is taken. This gives, for algebraic compactly generated triangulated categories, a characterisation of definable subcategories (see Section \ref{secdefsub}) which is analogous to \ref{chardefmod}.  This extends to any triangulated category with a suitable enhancement, see \cite[8.8]{SaoStovTstr} and \cite[6.8]{BirWilDualPr} which has the following formulation.

\begin{theorem}\label{LVchardef} \marginpar{LVchardef} (\cite[3.11]{LakDer}, \cite[4.7]{LakVit},\cite[8.8]{SaoStovTstr}, \cite[6.8]{BirWilDualPr}) If ${\cal D}$ is a subcategory of a compactly generated triangulated category ${\cal T}$ which is the underlying category of a strong and stable derivator, then the following are equivalent.

\noindent (i) ${\cal D}$ is a definable subcategory of ${\cal T}$;

\noindent (ii) ${\cal D}$ is closed in ${\cal T}$ under pure subobjects, products and directed homotopy colimits;

\noindent (iii) ${\cal D}$ is closed in ${\cal T}$ under pure subobjects, products and pure quotients.
\end{theorem}

Derived categories, derivators, dg-categories, model categories (in the sense of, say, \cite{HovBk}) and $\infty$-categories all provide ways of representing triangulated categories as the result of applying a process to a somewhat more amenable type of category.  In those additive categories with extra structure one can expect the model theory of (multisorted) modules to be directly applicable to the objects.  This gives the possibility of approaching the model theory of a triangulated category by developing model theory in such an enhancement and then passing this through a localisation-type functor to the triangulated category.  Examples include setting up elementary duality as done in \cite{AngHrbParam} and \cite{BirWilDualPr}, see Section \ref{secelemdual}.  We don't pursue this, so far relatively undeveloped, direction here.

\section{Definable subcategories}

\subsection{Definable subcategories of ${\cal T}$}\label{secdefsub} \marginpar{secdefsub}

A full subcategory ${\cal D}$ of ${\cal T}$ is {\bf definable} if its objects form the zero-class of a set of coherent functors, that is, if there is ${\cal A} \subseteq {\rm Coh}({\cal T})$ such that 
$${\cal D} = \{ X\in {\cal T}: FX=0 \, \,\forall F\in {\cal A}\}.$$
We will write ${\cal D} = {\rm Ann}({\cal A}) = {\rm Ann}_{\cal T}({\cal A}).$\footnote{We will also use this notation with a set of morphisms replacing ${\cal A}$ and hope this will not give rise to confusion.}  We will see in Section \ref{sectorsTc} how this is a natural extension of the notion of definable subcategory of a module category.
Also, if ${\cal X}$ is a subcategory of ${\cal T}$, set 
$${\rm Ann}_{{\rm Coh}({\cal T})}({\cal X}) = \{F\in {\rm Coh}({\cal T}) : FX=0 \,\, \forall X\in {\cal X}\}.$$
As for module categories, we denote by $\langle {\cal X} \rangle$ the definable subcategory of ${\cal T}$ {\bf generated by} ${\cal X}$ - that is, the smallest definable subcategory of ${\cal T}$ containing ${\cal X}$.

\vspace{4pt}

Given a set $\Phi$ of morphisms in ${\cal T}^{\rm c}$ we have its annihilator 
$${\rm Ann}_{\cal T}\Phi = \{ X\in {\cal T}: \forall A\xrightarrow{f} B \in \Phi, \,\,  \forall B \xrightarrow{b} X \mbox{ we have } bf=0\}.$$  
We write the condition $\forall B\, \xrightarrow{b} X \, (bf=0)$ succinctly as $Xf=0$ (this being directly analogous to the relation $Mr=0$ for a right module $M$ and ring element $r$).  Of course we can equally write this condition as $(f,X)=0$ or $(-,X)f=0$, according to our viewpoint.  Then, \cite[\S 7]{KraCoh}, ${\rm Ann}_{\cal T}\Phi$ is a (typical) definable subcategory of ${\cal T}$.

In the other direction, if ${\cal X}$ is a subcategory of ${\cal T}$, then we may set 
$${\rm Ann}_{{\cal T}^{\rm c}}{\cal X} = \{ A\xrightarrow{f}B \in {\cal T}^{\rm c}: Xf=0 \,\, \forall X\in {\cal X}\}.$$

The classes of morphisms of the form ${\rm Ann}_{{\cal T}^{\rm c}}{\cal X}$ are what Krause calls the {\bf cohomological ideals} of ${\cal T}^{\rm c}$; we will refer to them simply as {\bf annihilator ideals} in ${\cal T}^{\rm c}$.

\begin{lemma}\label{defcatann} \marginpar{defcatann} \cite[\S 7]{KraCoh} If $\Phi$ is a set of morphisms in ${\cal T}^{\rm c}$, then ${\rm Ann}_{\cal T}\Phi$ is a definable subcategory of ${\cal T}$.  If ${\cal X}$ is any subcategory of ${\cal T}$, then ${\rm Ann}_{\cal T}({\rm Ann}_{{\cal T}^{\rm c}}{\cal X}) = \langle {\cal X}\rangle$, the definable subcategory of ${\cal T}$ generated by ${\cal X}$.  In particular there is a natural bijection between the definable subcategories of ${\cal T}$ and the cohomological = annihilator ideals in ${\cal T}^{\rm c}$.
\end{lemma}

We have seen already that if
$$A\xrightarrow{f} B\xrightarrow{g} C \to \Sigma A$$ 
is a triangle, then 
$$bf=0 \quad \Leftrightarrow \quad g\mid b.$$
So we consider, given a set $\Psi$ of morphisms in ${\cal T}^{\rm c}$,
$${\rm Div}_{\cal T}\Psi = \{ X\in {\cal T}: \forall\, B\xrightarrow{g} C \in \Psi, \,\,  \forall\, B \xrightarrow{b} X, \, \exists \,C \xrightarrow{c} X \mbox{ such that } b=cg \}\quad \quad
\xymatrix{B \ar[r]^g \ar[d]_\forall & C \ar@{.>}[dl]^\exists \\
X
}$$
- the class of $\Psi$-divisible objects of ${\cal T}$.
We write $g|X$ as a succinct expression of the condition ``$\,\forall\,B \xrightarrow{b} X \, \exists \,C \xrightarrow{c} X \mbox{ such that } b=cg\,$" (being the analogue of the condition that every element of a module $M$ be divisible by an element $r$ of the ring\footnote{But the corresponding notation $Xg=X$ would be less appropriate than in the usual module case because $X$ has many sorts and that equation applies only to the $B$-sort of $X$.}).  Then ${\rm Div}_{\cal T}\Psi$ is a (typical) definable subcategory of ${\cal T}$.

And, in the other direction, given a  subcategory ${\cal X}$ of ${\cal T}$, we define\footnote{We are overworking the notations ${\rm Ann}$ and ${\rm Div}$ but they are useful.}
$${\rm Div}_{{\cal T}^{\rm c}}{\cal X} = \{ B\xrightarrow{g}C \in {\cal T}^{\rm c}: g|X \,\, \forall X\in {\cal X}\}.$$

\begin{lemma}\label{defcatdiv} \marginpar{defcatdiv} (\cite[2.2]{AngHrbParam}) If $\Psi$ is a set of morphisms in ${\cal T}^{\rm c}$, then ${\rm Div}_{\cal T}\Psi$ is a definable subcategory of ${\cal T}$.  If ${\cal X}$ is any subcategory of ${\cal T}$, then ${\rm Div}_{\cal T}({\rm Div}_{{\cal T}^{\rm c}}{\cal X}) = \langle {\cal X} \rangle$.
\end{lemma}
\begin{proof}  Take $Y\in {\rm Div}_{\cal T}({\rm Div}_{{\cal T}^{\rm c}}{\cal X})$.  If $g\in {\rm Div}_{{\cal T}^{\rm c}}{\cal X}$ then $g|Y$ so, if $f$ is as above, $Yf=0$.  This is so for all such $f$ (as $g$ varies) so, by \ref{defcatann}, $Y\in \langle {\cal X} \rangle$, as required.
\end{proof}

\begin{cor} \label{divandann} \marginpar{divandann} 
\noindent (1) If ${\cal D} = {\rm Ann}_{\cal T} \Phi$ is a definable subcategory of ${\cal T}$ then also ${\cal D} = {\rm Div}_{\cal T}\{g: A\xrightarrow{f}B \xrightarrow{g} C \to \Sigma A \text{ is a distinguished triangle and } f\in \Phi\}$.

\vspace{3pt}

\noindent (2) If ${\cal D} = {\rm Div}_{\cal T} \Psi$ is a definable subcategory of ${\cal T}$ then also ${\cal D} = {\rm Ann}_{\cal T}\{f: A\xrightarrow{f}B \xrightarrow{g} C \to \Sigma A \text{ is a distinguished triangle and } g\in \Psi\}$.
\end{cor}

Definable subcategories are so-called because they can be defined by closure of certain pairs of pp formulas, that is, by requiring that certain quotients of pp-definable subgroups be $0$.  For each of the annihilation and divisibility methods of specifying these subcategories, the pp-pairs needed are obvious, being respectively $\{ (x_B=x_B)/(x_Bf=0): f:A\to B \in \Phi\}$ and $\{ (x_B =x_B)/(g|x_B): g:B\to C \in \Psi\}$ with $\Phi$, $\Psi$ as above.

\vspace{4pt}

We have used that pp-pairs can be given in both annihilation and divisibility forms, but there is another, ``torsionfree" form that is not so obvious if we consider only formulas and their reduction to divisibility or annihilator forms, rather than pp-pairs.  Let us consider an extended triangle as before:
$$\Sigma^{-1}C \xrightarrow{\Sigma^{-1}h} A \xrightarrow{f} B \xrightarrow{g} C \xrightarrow{h} \Sigma A.$$
If $X\in {\cal T}$ then we obtain an exact sequence of abelian groups
$$(\Sigma A,X) \xrightarrow{(h,X)} (C,X) \xrightarrow{(g,X)} (B,X) \xrightarrow{(f,X)} (A,X) \xrightarrow{(\Sigma^{-1}h,X)} (\Sigma^{-1}C,X).$$
Then $X\in {\rm Div}_{\cal T}(g)$ iff $(g,X)$ is epi iff $(f,X)=0$ iff $(\Sigma^{-1}h,X)$ is monic.  If we denote by ${\rm ann}_X(\Sigma^{-1}h)$ the set $\{ a: A \to X: a.\Sigma^{-1}h=0\}$, then we have 
\begin{equation}\label{three}Xf=0 \quad \mbox{ iff } \quad g|X \quad \mbox{ iff } \quad {\rm ann}_X(\Sigma^{-1}h)=0.
\end{equation}
That is, $X\in {\cal T}$ annihilates $f$ iff it is $g$-divisible iff it is $\Sigma^{-1}h$-torsionfree.  This gives us a third way of using morphisms in ${\cal T}^{\rm c}$ to cut out definable subcategories of ${\cal T}$.  We set, given ${\cal X} \subseteq {\cal T}$
$${\cal X}\mbox{-}{\rm Reg} = \{ \ell\in {\cal T}^{\rm c}: {\rm ann}_X(\ell) =0 \,\, \forall X\in {\cal X}\}$$ 
and call such classes, for want of a better word, {\em regularity} classes (of morphisms of ${\cal T}^{\rm c}$).

In the other direction, given a set $\Xi$ of morphisms in ${\cal T}^{\rm c}$, we define 
$$\Xi\mbox{-}{\rm TF} = \{ X\in {\cal T}: {\rm ann}_X(\ell) =0 \,\, \forall \ell \in \Xi\}.$$

\begin{lemma}\label{defcatdiv1} \marginpar{defcatdiv1} If $\Xi$ is a set of morphisms in ${\cal T}^{\rm c}$, then $\Xi\mbox{-}{\rm TF}$ is a definable subcategory of ${\cal T}$.  If ${\cal X}$ is any subcategory of ${\cal T}$, then $({\cal X}\mbox{-}{\rm Reg})\mbox{-}{\rm TF} = \langle {\cal X} \rangle$.
\end{lemma}

The argument is as for \ref{defcatdiv}.

The set of pp-pairs corresponding to $\Xi$ is $\{ (x_A\ell=0)/(x_A=0): D\xrightarrow{\ell} A \in \Xi\}$.

The next result summarises some of this; see \cite[8.6]{SaoStovTstr} and, for the case where ${\cal T}$ is the derived category of modules over a ring,  \cite[2.2]{AngHrbParam}.

\begin{theorem}\label{defequivs} \marginpar{defequivs} A definable subcategory ${\cal D}$ of ${\cal T}$ may be specified by any of the following means:

\noindent ${\cal D} = \{ X\in {\cal T}: \phi(X)/\psi(X) =0 \,\, \forall \phi/\psi \in \Phi \}$ where $\Phi$ is a set of pp-pairs in ${\cal L}({\cal T})$;

\noindent ${\cal D} ={\rm Ann}_{\cal T}({\cal A})$ where ${\cal A} \subseteq {\rm Coh}({\cal T})$;

\noindent ${\cal D} = {\rm Ann}_{\cal T}\Phi$ where $\Phi$ is a set of morphisms in ${\cal T}^{\rm c}$;

\noindent ${\cal D} = {\rm Div}_{\cal T}\Psi$ where $\Psi$ is a set of morphisms in ${\cal T}^{\rm c}$;

\noindent ${\cal D} = \Xi\mbox{-}{\rm TF}$ where $\Xi$ is a set of morphisms in ${\cal T}^{\rm c}$.

The subcategories of ${\rm Coh}({\cal T})$ of the form ${\rm Ann}_{{\rm Coh}({\cal T})}({\cal D})$ are the Serre subcategories, the classes of morphisms of ${\cal T}^{\rm c}$ of the form ${\rm Ann}_{{\cal T}^{\rm c}}({\cal D})$ are the annihilator =  cohomological ideals\footnote{The classes ${\rm Div}_{{\cal T}^{\rm c}}{\cal D}$ and ${\cal D}\mbox{-}{\rm Reg}$ are described indirectly, in terms of the functors they present, at the end of Section \ref{secdefabs}.}.  

Moving between the last three specifications is described by (\ref{three}) above.
\end{theorem}

In Section \ref{secdefabs} we will say this in torsion-theoretic terms with ${\rm mod}\mbox{-}{\cal T}^{\rm c}$ in place of ${\rm Coh}({\cal T})$.  In Section \ref{sectorsTc} we give the relevant background.

\subsection{Torsion theories on ${\rm Mod}\mbox{-}{\cal T}^{\rm c}$}\label{sectorsTc} \marginpar{sectorsTc}

A {\bf torsion pair} in a Grothendieck category, such as ${\rm Mod}\mbox{-}{\cal T}^{\rm c}$, consists of two classes:  ${\cal G}$ - the {\bf torsion} class, and ${\cal F}$ - the {\bf torsionfree} class, with $({\cal G}, {\cal F})=0$ and with ${\cal G}$, ${\cal F}$ maximal such.  Such a torsion pair, or {\bf torsion theory}, is {\bf hereditary} if ${\cal G}$ is closed under subobjects, equivalently if ${\cal F}$ is closed under injective hulls and, if so, it is {\bf of finite type} if ${\cal G}$ is generated, as a hereditary torsion class, by finitely presented objects, equivalently if ${\cal F}$ is closed under directed colimits (see, for instance, \cite[11.1.12, 11.1.14]{PreNBK}).  We also use without further comment that, for a hereditary torsion theory, if $F$ is a torsionfree module then the injective hull $E(F)$ of $F$ is torsionfree (and conversely, since the torsionfree class is closed under subobjects).  For background on torsion theories, see \cite{Ste}.

The restricted Yoneda functor from ${\cal T}$ to ${\rm Mod}\mbox{-}{\cal T}^{\rm c}$ allows us to realise the definable subcategories of ${\cal T}$ as the inverse images of finite-type torsionfree classes on ${\rm Mod}\mbox{-}{\cal T}^{\rm c}$, as follows.

Suppose that ${\cal D}$ is a definable subcategory of ${\cal T}$.  Then ${\cal D}$ is determined by the class ${\cal D} \,\cap \,{\rm Pinj}({\cal T})$ of pure-injectives in it, being the closure of that class under pure subobjects (by the comments after \ref{chardefmod}).  By \ref{pinjtoinj} the image ${\cal E} = y({\cal D} \,\cap\, {\rm Pinj}({\cal T}))$ is a class of injective ${\cal T}^{\rm c}$-modules which is closed under direct products and direct summands, hence (e.g.~\cite[11.1.1]{PreNBK}) which is of the form ${\cal F} \cap {\rm Inj}\mbox{-}{\cal T}^{\rm c}$ for some hereditary torsionfree class ${\cal F} = {\cal F}_{\cal D}$ of ${\cal T}^{\rm c}$-modules.

We recall, \cite[3.3]{PreThes} see \cite[11.1.20]{PreNBK}, that a hereditary torsionfree class of modules is of finite type exactly if it is definable.  So we have to show that definability of ${\cal D}$ corresponds to definability of ${\cal F}_{\cal D}$, equivalently to definability of the class of absolutely pure objects in ${\cal F}_{\cal D}$ (``equivalently" because ${\rm Mod}\mbox{-}{\cal T}^{\rm c}$ is locally coherent, so the absolutely pure objects form a definable subcategory, see \cite[3.4.24]{PreNBK}, hence so is their intersection with any other definable subcategory; in the other direction, if ${\cal F}_{\cal D} \, \cap\, {\rm Abs}\mbox{-}{\cal T}^{\rm c}$ is definable then so also, by e.g.~\ref{chardefmod}, is its class of subobjects, which is precisely ${\cal F}_{\cal D}$).  So we have to show that the torsionfree class ${\cal F}_{\cal D}$ above is of finite type and that every finite type torsionfree class arises in this way.

To see, this, note that, if $X\in {\cal T}$ and $F\in {\rm Coh}({\cal T})$, then (Section \ref{secppsorts}) $FX =0$ iff $(F^\diamond, yX)=0$.  Set ${\cal A} = {\rm Ann}_{{\rm Coh}({\cal T})}({\cal D})$.  We have the duality from Section \ref{secppsorts} between ${\rm Coh}({\cal T})$ and ${\rm mod}\mbox{-}{\cal T}^{\rm c}$, so consider the corresponding set ${\cal A}^\diamond = \{ F^\diamond: F\in {\cal A}\}$ of finitely presented ${\cal T}^{\rm c}$-modules.  Since ${\cal A}$ is a Serre subcategory of ${\rm Coh}({\cal T})$, this is a Serre subcategory of ${\rm mod}\mbox{-}{\cal T}^{\rm c}$; we set ${\cal S}_{\cal D} = {\cal A}^\diamond$.  The $\varinjlim$-closure\footnote{If ${\cal S}$ is a subcategory of a module category, then we will denote its $\varinjlim${\bf -closure} - its closure under directed colimits - by $\overrightarrow{\cal S}$.}, $\overrightarrow{{\cal S}_{\cal D}}$, in ${\rm Mod}\mbox{-}{\cal T}^{\rm c}$ of ${\cal S}_{\cal D}$ is a typical hereditary torsion class of finite type in ${\rm Mod}\mbox{-}{\cal T}^{\rm c}$ (see \cite[11.1.36]{PreNBK}).  The corresponding hereditary torsionfree class ${\cal F}= \{ M\in {\rm Mod}\mbox{-}{\cal T}^{\rm c}: (\overrightarrow{{\cal S}_{\cal D}},M)=0\}$ is just the hom-perp of ${\cal S}_{\cal D}$:  ${\cal F} = \{ M\in {\rm Mod}\mbox{-}{\cal T}^{\rm c}: ({\cal S}_{\cal D},M)=0\}$.  If $M \in {\cal F}$ is injective, hence (\ref{pinjtoinj}) of the form $yN$ for some pure-injective $N\in {\cal T}$, then the condition $({\cal S}_{\cal D}={\cal A}^\diamond,M)=0$ is exactly the condition $FN=0$ for every $F \in {\cal A}$, that is, the condition that $N$ be in ${\cal D}$.  Thus ${\cal F} = {\cal F}_{\cal D}$ and we have the correspondence between classes of pure-injectives in ${\cal T}$ of the form ${\cal D}\, \cap\, {\rm Pinj}({\cal T})$ and classes of injectives in ${\rm Mod}\mbox{-}{\cal T}^{\rm c}$ of the form ${\cal F} \cap {\rm Inj}\mbox{-}{\cal T}^{\rm c}$ for some hereditary torsionfree class ${\cal F}$.  (For, note that given such a class ${\cal E}$ of injectives, the class of pure submodules of modules in ${\cal E}$ is the class of absolutely pure modules in ${\cal F}$ which, by finite type, is definable and hence has definable inverse image in ${\cal T}$).  Therefore we have shown the following.

\begin{theorem}\label{defcattf} \marginpar{defcattf}  A subcategory ${\cal D}$ of a compactly generated triangulated category ${\cal T}$ is definable iff it has any of the following equivalent forms, where $y:{\cal T} \to {\rm Mod}\mbox{-}{\cal T}^{\rm c}$ is the restricted Yoneda functor:

\noindent  ${\cal D} = y^{-1}{\cal F}$ where ${\cal F}$ is a finite-type hereditary torsionfree class in ${\rm Mod}\mbox{-}{\cal T}^{\rm c}$

\noindent ${\cal D} = y^{-1}{\cal E}$ where ${\cal E}$ is the class of absolutely pure objects in a hereditary torsionfree class of finite type;

\noindent  ${\cal D} =  y^{-1}{\cal E}$ where ${\cal E}$ is a definable class of absolutely pure objects in ${\rm Mod}\mbox{-}{\cal T}^{\rm c}$.
\end{theorem}

We denote by $\tau_{\cal D} = ({\mathscr T}_{\cal D}, {\cal F}_{\cal D})$ the finite-type hereditary torsion theory on ${\rm Mod}\mbox{-}{\cal T}^{\rm c}$ corresponding to ${\cal D}$.

\begin{cor}\label{deftf} \marginpar{deftf} The definable subcategories ${\cal D}$ of ${\cal T}$ are in natural bijection with the definable (= finite-type) hereditary torsionfree classes in ${\rm Mod}\mbox{-}{\cal T}^{\rm c}$ and also with the definable subcategories of ${\rm Abs}\mbox{-}{\cal T}^{\rm c}$.  

Explicitly, to ${\cal D}$ correspond respectively the closure ${\cal F}_{\cal D}$ of $\langle y{\cal D} \rangle$ under submodules, and ${\cal F}_{\cal D} \, \cap \, {\rm Abs}\mbox{-}{\cal T}^{\rm c}$.  In the other direction, we simply apply $y^{-1}$, where $y$ is the restricted Yoneda functor.
\end{cor}

Note the almost complete analogy of this with the bijection (see \cite[12.3.2]{PreNBK}) between definable subcategories of a module category ${\rm Mod}\mbox{-}R$ and the finite type (= definable) hereditary torsionfree classes in $(R\mbox{-}{\rm mod})\mbox{-}{\rm Mod} = (R\mbox{-}{\rm mod}, {\bf Ab})$, equivalently with the definable classes of absolutely pure objects in $(R\mbox{-}{\rm mod})\mbox{-}{\rm Mod} = (R\mbox{-}{\rm mod}, {\bf Ab})$.  One notable difference is that the image of a definable subcategory of a triangulated category is `most' of the definable subcategory $\langle y{\cal D} \rangle \, \subseteq \, {\rm Abs}\mbox{-}{\cal T}^{\rm c}$ of modules, whereas in the module case it is all of the corresponding class of modules.  This reflects the lack of directed colimits in triangulated categories, but see \cite{LakDer}, \cite{LakVit} for some replacement using Grothendieck derivators for the triangulated case.

The other notable difference is that the module case uses tensor product to embed (fully and faithfully) ${\rm Mod}\mbox{-}R$ in $(R\mbox{-}{\rm mod}, {\bf Ab})$.  Here we have somehow avoided that.

We also record the equivalence at the level of pure-injectives.

\begin{cor}\label{pinjtoinjD} \marginpar{pinjtoinjD} If ${\cal D}$ is a definable subcategory of ${\cal T}$ and ${\cal F}_{\cal D}$ is the corresponding hereditary torsionfree class in ${\rm Mod}\mbox{-}{\cal T}^{\rm c}$, then the restricted Yoneda functor $y$ induces an equivalence 
$${\rm Pinj}({\cal D}) \simeq {\cal F} \, \cap \,{\rm Inj}\mbox{-}{\cal T}^{\rm c}$$
between the category  ${\rm Pinj}({\cal D})$ of pure-injective objects of ${\cal T}$ which lie in ${\cal D}$ and the category $ {\cal F} \, \cap \,{\rm Inj}\mbox{-}{\cal T}^{\rm c}$ of ${\cal T}^{\rm c}$-injective modules which lie in ${\cal F}$.
\end{cor}

This gives some justification for our saying that the Yoneda image of a definable subcategory ${\cal D}$ in ${\rm Mod}\mbox{-}{\cal T}^{\rm c}$ constitutes `most of' the flat = absolutely pure objects of the corresponding hereditary torsionfree class of finite type.  For, every injective in the class is in the image and every absolutely pure object in the class is a pure (even elementary) submodule of an object in the image.

Note that the fact that the objects of ${\cal D}$ are the pure subobjects of the pure-injectives in ${\cal D}$ exactly corresponds to the fact that the absolutely pure modules in ${\cal F}$ are the pure submodules of the injective modules in ${\cal F}$.

\subsection{Definable subcategories of ${\rm Abs}\mbox{-}{\cal T}^{\rm c}$} \label{secdefabs} \marginpar{secdefabs}

In Section \ref{secdefsub} we associated to a definable subcategory ${\cal D}$ of ${\cal T}$ three sets of morphisms, ${\rm Ann}_{{\cal T}^{\rm c}}({\cal D})$, ${\rm Div}_{{\cal T}^{\rm c}}({\cal D})$ and ${\cal D}\mbox{-}{\rm Reg}$, each of which determines ${\cal D}$.  In this section we identify the corresponding sets of morphisms in ${\rm mod}\mbox{-}{\cal T}^{\rm c}$ and the ways in which they cut out the hereditary finite type torsion theory $\tau_{\cal D}$ cogenerated by $\langle y{\cal D}\rangle$ in ${\rm Mod}\mbox{-}{\cal T}^{\rm c}$.

We have the following from Section \ref{sectorsTc}.

\begin{cor}\label{defhered} \marginpar{defhered}  If ${\cal T}$ is a compactly generated triangulated category, then the following are in natural bijection:

\noindent (i) the definable subcategories of ${\cal T}$;

\noindent (ii) the definable subcategories of ${\rm Mod}\mbox{-}{\cal T}^{\rm c}$ which are contained in (so are definable subcategories of) ${\rm Abs}\mbox{-}{\cal T}^{\rm c} = {\rm Flat}\mbox{-}{\cal T}^{\rm c}$;

\noindent (iii) the hereditary torsion theories on ${\rm Mod}\mbox{-}{\cal T}^{\rm c}$ of finite type;

\noindent (iv) the Serre subcategories of ${\rm mod}\mbox{-}{\cal T}^{\rm c}$.
\end{cor}

Given a definable subcategory ${\cal D}$ of ${\cal T}$, let 
$${\cal S}_{\cal D} = \{ G\in {\rm mod}\mbox{-}{\cal T}^{\rm c}: (G,yX)=0 \,\, \forall \, X\in {\cal D}\}$$
be the corresponding Serre subcategory of ${\rm mod}\mbox{-}{\cal T}^{\rm c}$.  As noted in Section \ref{sectorsTc}, this is the Serre subcategory $({\rm Ann}_{{\cal T}^{\rm c}}({\cal D}))^\diamond$ of ${\rm mod}\mbox{-}{\cal T}^{\rm c}$, it $\varinjlim$-generates the finite type hereditary torsion class ${\mathscr T}_{\cal D}$ and $\tau_{\cal D} = ({\mathscr T}_{\cal D}, {\cal F}_{\cal D})$ is the torsion theory corresponding to ${\cal D}$ under (i)$\leftrightarrow$(iii) of \ref{defhered}.

If $\tau$ is any hereditary torsion theory then a submodule $L$ of a module $M$ is $\tau${\bf -dense in} $M$ if $M/L$ is torsion.  Also, the $\tau${\bf -closure}, ${\rm cl}_\tau^M(L)$, of a submodule $L$ of a module $M$ is the maximal submodule of $M$ in which $L$ is $\tau$-dense, also characterised as the smallest submodule $L'$ of $M$ which contains $L$ and is such that $M/L'$ is $\tau$-torsionfree. See \cite{Ste} or \cite[\S 11.1]{PreNBK} for details.

First we see that the annihilation, divisibility and regularity conditions with respect to ${\cal D}$ translate directly to ${\rm Mod}\mbox{-}{\cal T}^{\rm c}$.

\begin{prop}\label{translate} \marginpar{translate} Suppose that ${\cal D}$ is a definable subcategory of ${\cal T}$ and $f:A\to B$ is in ${\cal T}^{\rm c}$.  Then:

\noindent (1) $f\in {\rm Ann}_{{\cal T}^{\rm c}}({\cal D})$ iff $yX.yf=0$ for all $X\in {\cal D}$;

\noindent (2) $f \in {\rm Div}_{{\cal T}^{\rm c}}({\cal D})$ iff, for every $X\in {\cal D}$, $yX$ is $yf$-divisible;

\noindent (3) $f \in {\cal D}\mbox{-}{\rm Reg}$ iff, for every $X\in {\cal D}$, if $b':yB \to yX$ is such that $b'.yf=0$ then $b'=0$.
\end{prop}
\begin{proof}  First we note that, in all three cases, it is enough for the direction ($\Leftarrow$) to prove that $f$ has the property (annihilation, divisibility, regularity) for $X\in {\cal D}$ pure-injective.  That is because, if $X\in {\cal D}$, then $f$ satisfies, say, $Xf=0$ if (indeed iff) $H(X)f=0$, where $H(X)$ is the pure-injective hull of $X$.  That is because $X$ is pure in (indeed is an elementary substructure of) its pure-injective hull so, if a pp-pair is closed on $H(X)$, then it will be closed on $X$ (and {\it vice versa}). 

(1) The defining condition for $f$ to be in ${\rm Ann}_{{\cal T}^{\rm c}}({\cal D})$, namely that $Xf=0$ for all $X\in {\cal D}$, certainly implies $yX.yf=0$ for all $X\in {\cal D}$.    
If, conversely, $yX.yf=0$ for all $X \in {\cal D}$, then take $X\in {\cal D}$ and suppose we have $b:B\to X$.  Then $y(bf) =yb.yf=0$ so, by \ref{ybijcpct}, $bf=0$.  Therefore $Xf=0$, as required. 

(2) If $f \in {\rm Div}_{{\cal T}^{\rm c}}({\cal D})$ and we have $a':yA \to yX$, then we compose with the inclusion of $yX$ into its injective hull $E(yX)=yH(X)$ (by \ref{pinjtoinj}) to get a morphism $a'':yA \to yH(X)$ which, by \ref{ybijpi}, has the form $ya$ for some $a:A\to H(X)$.  By assumption, and since $H(X) \in {\cal D}$, $a$ factors through $f$, say $a = bf$ with $b:B\to H(X)$; therefore $a''=yb.yf$.
Thus $\exists x_{yB} (a'' = x_{yB}.yf)$ is true in $yH(X)$.  Since $yX$ is a pure submodule of $yH(X)$ we deduce that $\exists x_{yB} (a' = x_{yB}.yf)$ is true in $yX$, that is, $yX$ is $yf$-divisible.  This gives ($\Rightarrow$).

For the converse, suppose that, for every $X\in {\cal D}$, $yX$ is $yf$-divisible and take $X\in {\cal D}$ pure-injective and $a:A \to X$.  Then we have $ya: yA \to yX$ so, by hypothesis, there is $b':yB \to yX$ with $b'.yf = ya$.  Since $X$ is pure-injective, by \ref{ybijpi} there is $b:B \to X$ such that $b' = yb$, giving $y(bf) =ya$.  By \ref{ybijcpct} it follows that $bf=a$, showing that every pure-injective object in ${\cal D}$ is $f$-injective.  By the comments at the beginning of the proof and the fact that the divisibility condition is expressed by closure of a pp-pair, it follows that every object of ${\cal D}$ is $f$-injective, as required.

(3) The direction ($\Leftarrow$) follows immediately from \ref{ybijcpct}.  For the converse, if $f\in {\cal D}\mbox{-}{\rm Reg}$ then take $X\in {\cal D}$ to be pure-injective, and suppose $b':yB \to yX$ is such that $b'.yf=0$.  By \ref{ybijpi}, $b'=yb$ for some $b:B \to X$.  That gives $y(bf)=0$ hence, by \ref{ybijcpct}, $bf=0$, hence, by assumption, $b=0$, so that $b'=0$.  Thus $f$ is regular on every pure-injective in ${\cal D}$ and so, since that is expressed by closure of a pp-pair, $f$ is regular on every $X\in {\cal D}$, as required.
\end{proof}

Set ${\cal S}_{\cal D}^\circ = \{ G^\circ: G\in {\cal S}_{\cal D}\}$ to be the image of ${\cal S}_{\cal D} \subseteq {\rm mod}\mbox{-}{\cal T}^{\rm c}$ in ${\rm Coh}({\cal T})$ under the anti-equivalence \ref{cohfpdual}.  Note that, by definition of $G\mapsto G^\circ$, ${\cal S}_{\cal D}^\circ$ consists exactly of the coherent functors $F$ such that $FX=0$ for every $X\in {\cal D}$, that is $({\cal S}_{\cal D})^\circ = {\rm Ann}_{{\cal T}^{\rm c}}({\cal D})$.

\begin{prop} \label{torsequiv1} \marginpar{torsequiv1} Suppose that ${\cal D}$ is a definable subcategory of ${\cal T}$, let ${\cal S}_{\cal D}$ be the corresponding Serre subcategory of ${\rm mod}\mbox{-}{\cal T}^{\rm c}$.  Denote by $\tau_{\cal D}$ the corresponding hereditary (finite-type) torsion theory in ${\rm Mod}\mbox{-}{\cal T}^{\rm c}$.  Let $f:A\to B$ be a morphism in ${\cal T}^{\rm c}$.  Then the following hold.

\noindent (1)  $f\in {\rm Ann}_{{\cal T}^{\rm c}}({\cal D})$ iff ${\rm im}(yf) \in {\cal S}_{\cal D}$.

\noindent (2)  $f \in {\rm Div}_{{\cal T}^{\rm c}}({\cal D})$ iff ${\rm ker}(yf) \in {\cal S}_{\cal D}$ iff $F_f\in {\cal S}_{\cal D}^\circ$.

\noindent (3)  $f\in {\cal D}\mbox{-}{\rm Reg}$ iff $G_f ={\rm coker}(yf) \in {\cal S}_{\cal D}$, that is, iff ${\rm im}(yf)$ is $\tau_{\cal D}$-dense in $yB$.
\end{prop}
\begin{proof}  We use that $X\in {\cal D}$ iff $yX$ is ($\tau_{\cal D}$-)torsionfree, that is iff $({\cal S}_{\cal D}, yX)=0$.

(1) If the image ${\rm im}(yf)$ is in ${\cal S}_{\cal D}$ then, for every $X\in {\cal D}$, we have $({\rm im}(yf),yX)=0$ because $yX$ is torsionfree.  Therefore $yX.yf=0$, for all $X\in {\cal D}$ giving, by \ref{translate}, the implication ($\Leftarrow$).  For the other direction, first note that any morphism from ${\rm im}(yf)$ to $yX$ extends to a morphism from $yB$ to $yX$ by absolute purity = fp-injectivity of $yX$.  If ${\rm im}(yf)$ were not torsion, there would be a nonzero morphism from ${\rm im}(yf)$ to some torsionfree object which, for instance replacing the object by its injective hull, we may assume to be of the form $yX$ with $X\in {\cal D}$.  This would give a morphism $a: yB \to yX$ with $af\neq 0$, contradicting \ref{translate}.

(2) ($\Rightarrow$) By \ref{translate} we have that $yX$ is $yf$-divisible for every $X\in {\cal D}$.  If ${\rm ker}(yf)$ were not torsion (that is, since, by local coherence of ${\rm Mod}\mbox{-}{\cal T}^{\rm c}$, it is finitely presented, not in ${\cal S}_{\cal D}$) then it would have a nonzero torsionfree quotient $M$.  The (torsionfree) injective hull of $M$ would have the form $yX$ for some pure-injective $X\in {\cal D}$, yielding a morphism $yA \to yX$ which is not zero on the kernel of $yf$, hence which cannot factor through $yf$ - a contradiction.

For the converse, assume that ${\rm ker}(yf) \in {\cal S}_{\cal D}$.  Then any morphism $a':yA \to yX$ with $X\in {\cal D}$ must be zero on ${\rm ker}(yf)$, since $yX$ is torsionfree.  Therefore $a'$ factors through ${\rm im}(yf)$.  But $yX$ is absolutely pure so, since ${\rm im}(yf)$ is a finitely generated subobject of $yB$, that factorisation extends to a morphism $b':yB \to yX$.  Thus we have a factorisation of $a'$ through $yf$, and so $yX$ is $yf$-divisible.  By \ref{translate} that is enough.

For the part involving ${\cal S}_{\cal D}^\circ$, we have $f\in {\rm Div}_{{\cal T}^{\rm c}}({\cal D})$ iff $(f,X):(B,X) \to (A,X)$ is epi for every $X\in {\cal D}$ iff ${\rm coker}(f,X) =0$ for every $X\in {\cal D}$, that is, iff $F_fX=0$ for every $X\in {\cal D}$ and that, as noted above, is the case iff $F_f\in {\cal S}_{\cal D}^\circ$.

(3)  If ${\rm im}(yf)$ is not $\tau_{\cal D}$-dense in $yB$, there will be a nonzero morphism from $yB$ and with kernel containing ${\rm im}(yf)$ to a torsionfree object, hence to an object of the form $yX$ with $X\in {\cal D}$.  Therefore $yf$ is not $y{\cal D}$-regular and so, by \ref{translate}, $f$ is not ${\cal D}$-regular.

For the converse, suppose that ${\rm im}(yf)$ is $\tau_{\cal D}$-dense in $yB$.  Then, if $b'$ is a morphism from $yB$ to a torsionfree object and the kernel of $b'$ contains ${\rm im}(yf)$ then, since the image of $b'$ is torsion, we have $b'=0$.  Therefore every object in $y{\cal D}$ is $yf$-torsionfree which, by \ref{translate}, is as required.
\end{proof}

From this, \ref{defequivs} and the equivalences (4), we have the following, where we apply the notations ${\rm Ann}$, ${\rm Div}$ and ${\rm Reg}$ and their definitions to ${\rm Mod}\mbox{-}{\cal T}^{\rm c}$ with, of course, ${\rm mod}\mbox{-}{\cal T}^{\rm c}$ replacing ${\cal T}^{\rm c}$ as the subcategory of ``small" objects.  This is mostly \cite[5.1.4]{WagThesis}.

\begin{theorem}\label{morsinmodcat} \marginpar{morsinmodcat} Suppose that ${\cal D}$ is a definable subcategory of ${\cal T}$, let $\tau_{\cal D}$ be the corresponding finite-type hereditary torsion theory in ${\rm Mod}\mbox{-}{\cal T}^{\rm c}$ and let ${\cal S}_{\cal D}$ denote the Serre subcategory of $\tau_{\cal D}$-torsion finitely presented ${\cal T}^{\rm c}$-modules.  

Suppose that 
$$A\xrightarrow{f} B\xrightarrow{g} C \xrightarrow{h} \Sigma A$$ 
is a distinguished triangle.  Then: 

\noindent (i) $f\in {\rm Ann}_{{\cal T}^{\rm c}}({\cal D})$ iff $yf \in {\rm Ann}_{{\rm mod}\mbox{-}{\cal T}^{\rm c}}(y{\cal D})$ iff ${\rm im}(yf) \in {\cal S}_{\cal D}$;

\noindent (ii) $g\in {\rm Div}_{{\cal T}^{\rm c}}({\cal D})$ iff $yg \in {\rm Div}_{{\rm mod}\mbox{-}{\cal T}^{\rm c}}(y{\cal D})$ iff ${\rm ker}(yg) \in {\cal S}_{\cal D}$, that is, iff $F_g\in {\cal S}_{\cal D}^\circ$;

\noindent (iii) $\Sigma^{-1}h\in {\cal D}\mbox{-}{\rm Reg}$ iff the image of $y(\Sigma^{-1}h)$ is $\tau_{\cal D}$-dense in $y(\Sigma^{-1}C)$, that is, iff $G_{\Sigma^{-1}h}\in {\cal S}_{\cal D}$.

Furthermore, the conditions (i), (ii) and (iii) are equivalent.
\end{theorem}

\subsection{Model theory in definable subcategories}\label{secmoddef} \marginpar{secmoddef}

If ${\cal D}$ is a definable category, meaning a category equivalent to a definable subcategory of a module category (over a ring possibly with many objects), then the model theory of ${\cal D}$ is intrinsic to ${\cal D}$, in the following senses.  

First, the notion of pure-exact sequence is intrinsic to ${\cal D}$ because an exact sequence is pure-exact iff some ultraproduct of it is split-exact, see \cite[4.2.18]{PreNBK}.  Ultraproducts are obtained as directed colimits of products, so definable categories have ultraproducts.  Definable subcategories of compactly generated triangulated categories do not in general have directed colimits, so they are not (quite) ``definable categories" in this sense, though they are quite close, see \ref{LVchardef}.  Nevertheless, as we have seen, the restricted Yoneda functor associates, to a definable subcategory ${\cal D}$ of a compactly generated triangulated category, a definable subcategory of a module category which has the same model theory.

\vspace{4pt}

\noindent {\bf Question:}  Is the model theory of a definable subcategory ${\cal D}$ of a compactly generated triangulated category intrinsic, meaning definable just from the structure of ${\cal D}$ as a category?

\vspace{4pt}

Second, the category ${\mathbb L}({\cal D})^{\rm eq+}$ of pp-imaginaries for a definable subcategory ${\cal D}$ of a module category ${\rm Mod}\mbox{-}R$ is equivalent to the Serre localisation ${\mathbb L}^{\rm eq+}_R /{\cal S}_{\cal D}$, where ${\cal S}_{\cal D}$ is the Serre subcategory of coherent functors which annihilate ${\cal D}$.  We have the same description for a definable subcategory of a compactly generated triangulated category, {\it via} the restricted Yoneda functor.  But neither of those descriptions is intrinsic because both refer to a containing (module, or triangulated) category.  In the module case, there is an intrinsic description of ${\mathbb L}({\cal D})^{\rm eq+}$ as the category $({\cal D}, {\bf Ab})^{\prod \rightarrow}$ of functors from ${\cal D}$ to ${\bf Ab}$ which commute with direct products and directed colimits.  For ${\cal T}$ itself, there is a similar description in \cite[5.1]{KraCoh} but we may ask whether this extends to definable subcategories.

\vspace{4pt}

In any case, if ${\cal D}$ is a definable subcategory of a compactly generated triangulated category ${\cal T}$, then the category, ${\mathbb L}({\cal D})^{\rm eq+}$, of pp-imaginaries for ${\cal D}$ is the quotient of ${\mathbb L}({\cal T})^{\rm eq+}$ by its Serre subcategory consisting of those pp-pairs which are closed on ${\cal D}$.  In terms of the other forms of the category of pp-imaginaries given by \ref{ppcat}, ${\mathbb L}({\cal D})^{\rm eq+}$ also has the following descriptions.  

\begin{prop}\label{LeqD} \marginpar{LeqD} If ${\cal D}$ is a definable subcategory of a compactly generated triangulated category ${\cal T}$, then the following categories are equivalent:

\noindent (i) the category, ${\mathbb L}({\cal D})^{\rm eq+}$, of pp-imaginaries for ${\cal D}$;

\noindent (ii) ${\rm Coh}({\cal T})/{\rm Ann}_{{\rm Coh}({\cal T})}({\cal D})$;

\noindent (iii) ${\rm mod}\mbox{-}{\cal T}^{\rm c}/{\cal S}_{\cal D}$.
\end{prop}

Note that the contravariant action of ${\mathbb L}({\cal T})^{\rm eq+}$ {\it via} $({\mathbb L}({\cal T})^{\rm eq+})^{\rm op} \simeq {\rm mod}\mbox{-}{\cal T}^{\rm c}$ acting by $G(X) = (G,yX)$ for $G\in {\rm mod}\mbox{-}{\cal T}^{\rm c}$ and $X\in {\cal T}$ localises as the action of ${\rm mod}\mbox{-}{\cal T}^{\rm c}/{\cal S}_{\cal D}$ on $ \langle Q_{\cal D}(y{\cal D}) \rangle = \langle Q_{\cal D}(y{\cal T}) \rangle$ where $Q_{\cal D}: {\rm Mod}\mbox{-}{\cal T}^{\rm c} \to {\rm Mod}\mbox{-}{\cal T}^{\rm c}/\overrightarrow{{\cal S}_{\cal D}}$ is the corresponding Gabriel localisation and the action is given by the same formula.  This places both the category of models and the category of imaginaries (the latter contravariantly) into the same Grothendieck abelian category, just as in the module case where we can use the tensor embedding, see \cite[\S 12.1.1]{PreNBK}.

\subsection{Hom-orthogonal pairs on ${\cal T}$ and torsion theories on ${\rm Mod}\mbox{-}{\cal T}^{\rm c}$}\label{sectors} \marginpar{sectors}

A {\bf hom-orthogonal pair}\footnote{In the context of triangulated categories, the term ``torsion pair" is used for a stronger concept, see \cite[\S 3]{StoPop}.} on ${\cal T}$ is a pair $({\cal U}, {\cal V})$ of subcategories with ${\cal U} = {}^\perp{\cal V}$ the {\bf torsion} class and ${\cal V} = {\cal U}^\perp$ the {\bf torsionfree} class.  Such a pair $({\cal U}, {\cal V})$ is said to be {\bf compactly generated} if there is ${\cal A} \subseteq {\cal T}^{\rm c}$ such that ${\cal V} = {\cal A}^\perp = \{ Y\in {\cal T}: (A, Y)=0 \, \, \forall A\in {\cal A} \}$, in which case ${\cal U} = {}^\perp({\cal A}^\perp) = \{Z\in {\cal T}: (Z,{\cal A}^\perp)=0\}$; we say that ${\cal A}$ {\bf generates} the hom-orthogonal pair.  Note that ${\cal V}$ is in this case definable, being given by the conditions that each sort $(A,-)$ for $A\in {\cal A}$ is $0$, that is, all the pp-pairs $x_A=x_A/x_A=0$ for $A\in {\cal A}$ are closed on ${\cal V}$.

\begin{prop}  Suppose that $({\cal U}, {\cal V})$ is a hom-orthogonal pair in ${\cal T}$, compactly generated by ${\cal A} \subseteq {\cal T}^{\rm c}$.  Let $\tau_{\cal V} = ({\mathscr T}_{\cal V}, {\cal F}_{\cal V})$ denote the finite-type hereditary torsion theory on ${\rm Mod}\mbox{-}{\cal T}^{\rm c}$ corresponding (\ref{deftf}) to the definable subcategory ${\cal V}$.  Let ${\rm Ser}(y{\cal A})$ denote the Serre subcategory of ${\rm mod}\mbox{-}{\cal T}^{\rm c}$ generated by $y{\cal A}$.  

Then ${\mathscr T}_{\cal V} = \overrightarrow{{\rm Ser}(y{\cal A})}$ and ${\cal F}_{\cal V} = (y{\cal A})^\perp =  \{M\in {\rm Mod}\mbox{-}{\cal T}^{\rm c}: (yA,M)=0 \,\, \forall A\in {\cal A}\}$.
\end{prop}
\begin{proof}  This follows from what we have seen already; we give the details.  Since $({\cal A}, {\cal V})=0$, it follows by \ref{ybijcpct} that $(y{\cal A}, y{\cal V})=0$, so $\overrightarrow{{\rm Ser}(y{\cal A})} \subseteq {\mathscr T}_{\cal V}$.  Hence ${\cal F}_{\cal V} = ({\mathscr T}_{\cal V})^\perp \subseteq (\overrightarrow{{\rm Ser}(y{\cal A})})^\perp = (y{\cal A})^\perp$ (equality since $\tau_{\cal V}$ is of finite type).  If, conversely, $M \in (y{\cal A})^\perp$, then so is $E(M)$, which has the form $yN$ for some pure-injective $N\in {\cal T}$.  By \ref{ybijpi} (or \ref{ybijcpct}), $({\cal A}, N)=0$ and hence $N\in {\cal V}$, so $E(M)$, and hence $M$ is in ${\cal F}_{\cal V}$.  Thus ${\cal F}_{\cal V} = (y{\cal A})^\perp$ and hence also ${\mathscr T}_{\cal V} = \overrightarrow{{\rm Ser}(y{\cal A})}$.
\end{proof}

By \ref{deftf}, every finite-type hereditary torsion theory $({\mathscr T}, {\cal F})$ on ${\rm Mod}\mbox{-}{\cal T}^{\rm c}$ gives rise to a hom-orthogonal pair in ${\cal T}$, namely $({}^\perp{\cal D}, ({}^\perp{\cal D})^\perp)$ where ${\cal D} = y^{-1}{\cal F}$.  If this hom-orthogonal pair is compactly generated, by ${\cal A}$ say, so $ ({}^\perp{\cal D})^\perp = {\cal A}^\perp$ is definable, then it follows from the above that ${\cal F} = (y{\cal A})^\perp$ and hence ${\cal D}  = y^{-1}{\cal F} = y^{-1}((y{\cal A})^\perp) = {\cal A}^\perp$ (by the bijection \ref{deftf}) $=({}^\perp{\cal D})^\perp$.  But in general not every finite-type hereditary torsion class in ${\rm Mod}\mbox{-}{\cal T}^{\rm c}$ arises from a hom-orthogonal pair in ${\cal T}$ in this way.  Indeed, since, for $A \in {\cal T}^{\rm c}$, $yA$ is a projective ${\cal T}^{\rm c}$-module, and all of the finitely generated projectives in ${\rm Mod}\mbox{-}{\cal T}^{\rm c}$ are of this form, we have the following, where we denote by $\gamma_{\cal X}$ the hereditary (finite type) torsion theory generated by (that is, with torsion class generated by) $y{\cal X}$.

\begin{cor}\label{cpctgentors} \marginpar{cpctgentors} There is a natural injection $({\cal U},{\cal V}) \mapsto \gamma_{\cal U}$ from the set of compactly generated hom-orthogonal pairs in ${\cal T}$ to the set of hereditary torsion theories of finite type on ${\rm Mod}\mbox{-}{\cal T}^{\rm c}$.

The image is the set of hereditary torsion theories where the torsion class is generated by a set of finitely generated projectives.
\end{cor}

Thus we have an embedding of the lattice of compactly generated hom-orthogonal pairs in ${\cal T}$ into the lattice of finite type hereditary torsion theories on ${\rm Mod}\mbox{-}{\cal T}^{\rm c}$ (the ordering in each case being by inclusion of torsion classes), and the latter is isomorphic to the lattice of definable subcategories of ${\cal T}$.  The definable subcategories, ${\cal D}$, of ${\cal T}$ occurring as ${\cal V}$ in a compactly generated hom-orthogonal pair $({\cal U}, {\cal V})$, are, by \ref{torsequiv1}(1), those for which the corresponding annihilator ideal ${\rm Ann}_{{\cal T}^{\rm c}}({\cal D})$ of ${\cal T}^{\rm c}$ is generated as such by objects (that is, by identity morphisms of some compact objects).

Note also that, if ${\cal D}$ is a definable subcategory of ${\cal T}$ which occurs as ${\cal V}$ in a compactly generated hom-orthogonal pair $({\cal U}, {\cal V})$, and if $({\mathscr T}, {\cal F})$ is the corresponding, in the sense of \ref{deftf}, torsion theory $\tau_{\cal D}$, then we always have ${\cal U} \subseteq y^{-1}{\mathscr T}$.  That is because ${\mathscr T} = {^\perp}({\cal F} \cap {\rm Inj}\mbox{-}{\cal T}^{\rm c})$ and because each object of ${\cal F} \cap {\rm Inj}\mbox{-}{\cal T}^{\rm c}$ has the form $yN$ for some pure-injective $N\in {\cal V}$ and then $({\cal U},N) =0$ implies, by \ref{ybijpi}, that $(y{\cal U}, yN)=0$, so $y{\cal U} \subseteq {\mathscr T}$.  For equality, ${\cal U} \subseteq y^{-1}{\mathscr T}$ - that is $\gamma_{\cal U} = \tau_{\cal D}$ - we need, by the argument just given, that ${\cal U} = {^\perp}({\cal V}\cap {\rm Pinj}({\cal T}))$.  That is, equality holds iff the hom-orthogonal pair $({\cal U}, {\cal V})$ is {\bf cogenerated by pure-injectives}.  For instance, if $({\cal U}, {\cal V})$ is a t-structure with ${\cal V}$ definable, then this will be the case, \cite[2.10]{AngHrbParam}, \cite[8.20]{SaoStovTstr}, also see \ref{smash2} below.

For more about this and TTF-classes in compactly generated triangulated categories, see \cite[Chpt.~8]{WagThesis}.

\subsection{Spectra}\label{secspec} \marginpar{secspec}

By a definable (additive) category we mean a category which is equivalent to a definable subcategory of the category of modules over some (possibly multi-sorted) ring.  Every definable additive category ${\cal C}$ is determined by its full subcategory of pure-injective objects (by \cite[5.1.4]{PreNBK} or, more intrinsically, by \cite[\S 3.2]{PreDefAddCat}).  Indeed, every definable category is determined by the indecomposable pure-injective objects in it (e.g. see \cite[5.3.50, 5.3.52]{PreNBK}).  The Ziegler spectrum, ${\rm Zg}({\cal C})$, also written ${\rm Zg}_R$ in the case ${\cal C} = {\rm Mod}\mbox{-}R$, is the set, ${\rm pinj}({\cal C})$, of isomorphism classes of indecomposable pure-injectives in ${\cal C}$ endowed with the topology which has, for a basis of open sets, the 
$$(\phi/\psi) = \{ N\in {\rm pinj}({\cal C}): \phi(N) > \psi(N)\}$$
as $\phi/\psi$ ranges over pp-pairs (in any suitable language for ${\cal C}$).  These are exactly the compact open sets in ${\rm Zg}({\cal C})$, see \cite[5.1.22]{PreNBK}.

Every definable subcategory ${\cal D}$ of a definable category ${\cal C}$ is determined by the set ${\rm pinj}({\cal D}) = {\cal D} \, \cap \, {\rm pinj}({\cal C})$ of indecomposable pure-injectives in ${\cal D}$, hence by the closed subset ${\rm Zg}({\cal D}) = {\cal D} \cap {\rm Zg}({\cal C})$ of ${\rm Zg}({\cal C})$, and every closed set in ${\rm Zg}({\cal C})$ is of the form ${\rm Zg}({\cal D})$ for some definable subcategory ${\cal D}$ of ${\cal C}$, see \cite[5.1.1]{PreNBK}.  

Krause \cite{KraTel} showed how this carries over to compactly generated triangulated categories ${\cal T}$.  The {\bf Ziegler spectrum}, ${\rm Zg}({\cal T})$, of ${\cal T}$ is defined to have, for its points, the (isomorphism classes of) indecomposable pure-injectives.  As for definable subcategories of module categories, there are many equivalent ways of specifying a basis of (compact) open sets on this set of points, including the following (the second by \ref{ppprfmla}):

\noindent $(\phi/\psi) =\{ N\in {\rm pinj}({\cal T}): \phi(N)/\psi(N) \neq 0\}$ for $\phi/\psi$ a pp-pair;

\noindent $\{ N\in {\rm pinj}({\cal T}): {\rm ann}_N(f) \neq 0\}$ for $f$ a morphism in ${\cal T}^{\rm c}$;

\noindent $(F) =\{ N\in {\rm pinj}({\cal T}): FN \neq 0\}$ for $F\in {\rm Coh}({\cal T})$.

\vspace{4pt}

There are other topologies of interest here.  First consider the case where $R$ is commutative noetherian.  Then the subcategory, ${\rm Inj}\mbox{-}R$, of injectives in ${\rm Mod}\mbox{-}R$ is definable (see \cite[3.4.28]{PreNBK}) and the corresponding closed subset of ${\rm Zg}_R$ is just the set, ${\rm inj}_R$, of indecomposable injective $R$-modules.  For such a ring the set ${\rm inj}_R$ may be identified \cite{Gab}, see \cite[\S 14.1.1]{PreNBK}, with ${\rm Spec}(R)$ {\it via} $P \mapsto E(R/P)$ where $P$ is any prime ideal of $R$ and $E(-)$ denotes injective hull.  However, the Ziegler topology restricted from ${\rm Zg}_R$ to ${\rm inj}_R$ induces, {\it via} the above bijection, not the Zariski topology on ${\rm Spec}(R)$ but its Hochster dual (\cite[pp.~104/5]{PreBk}).  Recall that the {\bf Hochster dual} of a topology has, as a basis (on the same set of points), the complements of the compact open sets in the original topology.

That fact inspired the general definition \cite[pp.~200-202]{PreRem} of the {\bf dual-Ziegler} (or ``rep-Zariski") topology on ${\rm pinj}({\cal C})$ for any definable category ${\cal C}$, as the Hochster-dual of the Ziegler topology\footnote{These spaces are, however, unlike those in Hochster's original definition, not spectral, and it is not always that case that the Ziegler topology is returned as the dual of the dual-Ziegler topology, \cite[3.1]{BurPre2}}.  So this dual topology has the same underlying set, ${\rm pinj}({\cal C})$, and has, for a basis of open sets, the complements 
$$[\phi/\psi] = {\rm Zg}({\cal C}) \setminus (\phi/\psi)$$
of the compact Ziegler-open sets. 

If ${\cal C}$ is a locally coherent category, in particular if it is ${\rm Mod}\mbox{-}R$ for a right coherent ring (possibly with many objects), then\footnote{For module categories, this goes back to \cite{EkSab}, see \cite[3.4.24]{PreNBK}; the general case is proved the same way and also follows from, for example, \cite[Chpt.~6]{PreMAMS}.} the absolutely pure objects form a definable subcategory with corresponding closed subset of ${\rm Zg}({\cal C})$ again being the set ${\rm inj}({\cal C})$ of (isomorphism types of) indecomposable injectives in ${\cal C}$.  This set carries a ({\bf Gabriel}-){\bf Zariski} topology which has, for a basis of open sets, those of the form
$$[A]= \{ E\in {\rm inj}({\cal C}): (A,C)=0\}$$
for $A$ a finitely presented object of ${\cal C}$.  Thus we extend the domain of applicability of the category-theoretic reformulation (\cite{Gab}, \cite{Roolim}) of the definition of the Zariski topology on a commutative coherent ring.   For such a category ${\cal C}$ the Gabriel-Zariski topology coincides with the dual-Ziegler topology restricted to ${\rm inj}({\cal C})$ (\cite[14.1.6]{PreNBK}).

\vspace{4pt}

We may compare these topologies over a commutative coherent ring $R$ where, in general, the map $P\mapsto E(R/P)$ is only an inclusion of ${\rm Spec}(R)$ into ${\rm inj}_R$, because there may be indecomposable injectives not of the form $E(R/P)$, e.g.~\cite[14.4.1]{PreNBK}.  The inclusion, nevertheless, is a topological equivalence - an isomorphism of frames of open subsets:  every indecomposable injective is elementarily equivalent to, hence topologically equivalent to, a module of the form $E(R/P)$ with $P$ a prime, see \cite[14.4.5]{PreNBK}.  So, for commutative coherent rings, we may consider these various topologies as topologies on ${\rm Spec}(R)$ and, so considered, the Ziegler topology coincides with the {\bf Thomason} topology, which is defined to be the Hochster-dual of the Gabriel-Zariski topology, \cite{GarkPre3}.  That is, the Ziegler topology has, for its open sets, those of the form $\bigcup_\lambda \, (R/I_\lambda)$ with the $I_\lambda$ finitely generated ideals of $R$, where
$$(R/I_\lambda) = \{ N \in {\rm pinj}_R: (R/I_\lambda, N) \neq 0\} = (xI_\lambda=0/x=0).$$
In terms of sets of primes, the Ziegler-open sets have the form $\bigcup_\lambda \, V(I_\lambda)$ with the $I_\lambda$ finitely generated\footnote{For a general commutative ring, the Ziegler topology on ${\rm inj}_R$ is finer, having open sets of a similar form but with pp-definable ideals replacing finitely generated ideals; in coherent rings the pp-definable ideals coincide with the finitely generated ideals, see \cite[\S 6]{PreSpecAb}.}.  These various topologies are compared in \cite[\S 6]{PreSpecAb}.

\vspace{4pt}

The discussion above applies to the locally coherent category ${\rm Mod}\mbox{-}{\cal T}^{\rm c}$.  As we have seen in \ref{pinjtoinj}, the restricted Yoneda functor $y$ induces an equivalence between the category, ${\rm Pinj}({\cal T})$, of pure-injective objects of ${\cal T}$ and the category, ${\rm Inj}\mbox{-}{\cal T}^{\rm c}$, of injective right ${\cal T}^{\rm c}$-modules.  Indeed, this gives a homeomorphism of spectra.

\begin{theorem}\label{pinjinj} \marginpar{pinjinj} Suppose that ${\cal T}$ is a compactly generated triangulated category.  Then $y:{\cal T} \to {\rm Mod}\mbox{-}{\cal T}^{\rm c}$ induces a bijection between ${\rm pinj}({\cal T})$ and ${\rm inj}{{\cal T}^{\rm c}}$. This is a homeomorphism between ${\rm Zg}({\cal T})$ and ${\rm Zg}({\rm Abs}\mbox{-}{\cal T}^{\rm c} = {\rm Flat}\mbox{-}{\cal T}^{\rm c})$ (the latter can also be regarded as ${\rm inj}_{{\cal T}^{\rm c}}$ with the Thomason topology) and is also a homeomorphism between the dual-Ziegler spectrum ${\rm Zar}({\cal T})$ of ${\cal T}$ and ${\rm inj}_{\cal T}^{\rm c}$ if the latter is equipped with the Gabriel-Zariski topology which has, for a basis of open sets, the sets $[G] = \{ E\in {\rm Inj}\mbox{-}{\cal T}^{\rm c} : (G,E)=0\}$ for $G\in {\rm mod}\mbox{-}{\cal T}^{\rm c}$.
\end{theorem}

Since closed subsets of the Ziegler spectrum are in natural correspondence with definable subcategories, this homeomorphism underlies the bijection \ref{deftf} between definable subcategories of ${\cal T}$ and finite-type hereditary torsionfree classes in ${\rm Mod}\mbox{-}{\cal T}^{\rm c}$.  That also reflects the fact that a finite-type hereditary torsion theory is determined by (it is the torsionfree class cogenerated by) the set of indecomposable torsionfree injectives (see \cite[11.1.29]{PreNBK}).  We have already, in Section \ref{sectors}, considered the part of this correspondence coming from compactly generated hom-orthogonal pairs in ${\cal T}$, and we will also, in Section \ref{secttspec}, look at how the Balmer spectrum fits into this picture in the case that ${\cal T}$ is tensor-triangulated.

\subsection{Triangulated definable subcategories}\label{sectriangdef} \marginpar{sectriangdef}

In this section we consider the definable subcategories ${\cal D}$ of ${\cal T}$ which are {\bf triangulated}, that is, {\bf shift-closed} (if $X\in {\cal D}$, then $\Sigma^\pm X \in {\cal D}$) and extension-closed, where by {\bf extension-closed} we mean that, if $X\to Y\to Z \to \Sigma X$ is a distinguished triangle with both $X$ and $Z$ in ${\cal D}$, then also $Y\in {\cal D}$.  First, some remarks on extending definable subcategories to shift-closed definable subcategories.

\vspace{4pt}

If ${\cal D}$ is a definable subcategory of ${\cal T}$ then each shift $\Sigma^i {\cal D}$ is definable, (e.g.~see \cite[6.1.1]{WagThesis}).  We can define the shift-closure of ${\cal D}$ to be the definable closure of $\bigcup_{i\in {\mathbb Z}} \, \Sigma^i {\cal D}$.  That this is, in general, larger than ${\rm Add}^+(\bigcup_{i\in {\mathbb Z}} \, \Sigma^i {\cal D})$ ($^+$ denoting closure under pure submodules) is shown by the following example.

\begin{example}\label{kepsilonshift} \marginpar{kepsilonshift}  Consider the derived category ${\cal D}_{k[\epsilon]} = {\cal D}({\rm Mod}\mbox{-}k[\epsilon])$, of the category of modules over $k[\epsilon]= k[x]/(x^2)$.  Let ${\cal D}$ be the subcategory of ${\cal D}_{k[\epsilon]}$ consisting of complexes which are $0$ in every degree $i < 0$.  Then ${\cal D}$ is a definable subcategory, defined by the conditions $(k[\epsilon][i], -)=0$ ($i<0$) where $k[\epsilon]$ here denotes the complex with $k[\epsilon]$ in degree 0 and zeroes elsewhere.

The union of the (left) shifts of ${\cal D}$ contains only complexes which are bounded below and so the additive closure of the union $\bigcup_i \, {\rm Zg}(\Sigma^i {\cal D})$ of the Ziegler-spectra of these shifts does not contain, for example, the doubly infinite complex which has $k[\epsilon]$ in each degree and multiplication by $\epsilon$ for each of its maps.  But that indecomposable pure-injective complex belongs to the Ziegler-closure of $\bigcup_i \, {\rm Zg}(\Sigma^i {\cal D})$, indeed it is in the Ziegler-closure of the set of complexes obtained from it by replacing $k[\epsilon]$ by $0$ in every degree $\leq i$ for some $i$; this is proved in \cite[\S 3.4]{HanThesis} and, in greater generality, in \cite[\S 6, \S 4]{ALPP}.

In contrast, if we were to take ${\cal D}$ to be the image of ${\rm Mod}\mbox{-}k[\epsilon]$ consisting of complexes concentrated in degree $0$, then the additive closure of the union of the shifts of ${\cal D}$ is definable.  That follows because every object in the definable category generated by that union is finite endolength, so the Ziegler closure contains no new indecomposable pure-injectives (e.g.~see \cite[4.4.30]{PreNBK}).
\end{example}

Thus, if $X$ is a closed subset of the Ziegler spectrum of ${\cal T}$, it may be that $\bigcup_i \, \Sigma^iX$ is not Ziegler-closed.

It is the case, see \cite[6.1.10]{WagThesis}, that, if points of ${\rm Zg}({\cal T})$ are identified with their shifts and the set of equivalence classes is given the quotient topology, then this is topologically equivalent to the space based on ${\rm pinj}({\cal T})$ which has, for its closed sets, those of the form ${\cal D} \cap {\rm pinj}({\cal T})$ where ${\cal D}$ is a shift-closed definable subcategory of ${\cal T}$.  The first example in \ref{kepsilonshift} shows that the projection map taking a point of the Ziegler spectrum of ${\cal T}$ to its shift equivalence class need not be closed (the complexes in that example are endofinite, hence Ziegler-closed points).

Further Ziegler-type topologies on ${\rm pinj}({\cal T})$ are obtained by using positively- (alternatively, negatively-) shift-closed definable subcategories of ${\cal T}$, see \cite[\S 6.1]{WagThesis}).

\vspace{4pt}

A triangulated subcategory ${\cal B}$ of ${\cal T}$ is {\bf smashing} if it is the kernel of a Bousfield localisation $q:{\cal T} \to {\cal T}'$ for which the left adjoint to $q$, including ${\cal T}' = {\cal T}/{\cal B}$ into ${\cal T}$, preserves coproducts.  Hom-orthogonality gives a bijection between the definable subcategories which are triangulated and the smashing subcategories of ${\cal T}$.

\begin{theorem}\label{smash1} \marginpar{smash1} (\cite{KraCohQuot}, see \cite[5.2.10]{WagThesis}) If ${\cal D}$ is a triangulated definable subcategory of the compactly generated triangulated category ${\cal T}$, then ${\cal B}=\, ^\perp {\cal D}$ is a smashing subcategory of ${\cal T}$ and ${\cal D} = {\cal B}^\perp$, so $({\cal B}, {\cal D})$ is a torsion pair.  Every smashing subcategory of ${\cal T}$ arises in this way.
\end{theorem}

\begin{prop}\label{smash2} \marginpar{smash2} \cite[3.9, Thm.~C]{KraTel}  Suppose that ${\cal B}$ is a smashing subcategory of ${\cal T}$ and ${\cal D}= {\cal B}^\perp$ is the corresponding triangulated definable subcategory.  Then ${\cal B} = y^{-1}{\mathscr T}_{\cal D}$, where ${\mathscr T}_{\cal D} = \overrightarrow{{\cal S}_{\cal D}}$ is the torsion class for the torsion theory $\gamma_{\cal B} = \tau_{\cal D}$ generated by $y{\cal B}$, equivalently cogenerated by $y{\cal D}$.
\end{prop}

\begin{cor}\label{smash3} \marginpar{smash3}  If ${\cal D}$ is a triangulated definable subcategory of ${\cal T}$, and ${\mathscr T}_{\cal D}$ is the corresponding hereditary torsion class in ${\rm Mod}\mbox{-}{\cal T}^{\rm c}$, then $y^{-1}{\mathscr T}_{\cal D} =\, ^\perp{\cal D}$ is a (typical) smashing subcategory of ${\cal T}$.
\end{cor}

One says that ${\cal T}$ has the {\bf Telescope Property} if, for each smashing subcategory ${\cal B}$, the torsion pair $({\cal B}, {\cal D})$ is compactly generated, equivalently, \ref{cpctgentors}, if the Serre subcategory ${\cal S}_{\cal D} = {\mathscr T}_{\cal D} \, \cap \, {\rm mod}\mbox{-}{\cal T}^{\rm c}$ is generated by projective (= representable) objects, see \cite[Introduction]{KraTel}.

\subsection{Elementary duality}\label{secelemdual} \marginpar{secelemdual}

If $R$ is any skeletally small preadditive category (= multisorted ring), then there is a duality - {\em elementary duality}, \cite{PreDual}, \cite{HerzDual}, see \cite[\S \S 1.3, 10.3]{PreNBK} - between the category of pp-pairs for right $R$-modules and the category of pp-pairs for left $R$-modules.  This duality induces a natural bijection between the definable subcategories of ${\rm Mod}\mbox{-}R$ and $R\mbox{-}{\rm Mod}$, \cite[6.6]{HerzDual} see \cite[\S 3.4.2]{PreNBK}.

In particular this applies with $R={\cal T}^{\rm c}$.  Because the model theory of ${\cal T}$ is essentially that of ${\rm Flat}\mbox{-}{\cal T}^{\rm c} = {\rm Abs}\mbox{-}{\cal T}^{\rm c}$ inside ${\rm Mod}\mbox{-}{\cal T}^{\rm c}$, it follows that we have a version of elementary duality between ${\cal T}$ and the definable subcategory ${\cal T}^{\rm c}\mbox{-}{\rm Abs} = {\cal T}^{\rm c}\mbox{-}{\rm Flat}$ of ${\cal T}^{\rm c}\mbox{-}{\rm Mod}$.  In particular, elementary duality gives a natural bijection between the definable subcategories of ${\cal T}$ and those of ${\cal T}^{\rm c}\mbox{-}{\rm Flat}$.

With the module situation in mind, it is natural to ask whether there is a compactly triangulated category ${\cal T}_1$ such that ${\cal T}_1^{\rm c} \simeq ({\cal T}^{\rm c})^{\rm op}$ and hence an elementary duality between the model theory of ${\cal T}$ and the model theory of ${\cal T}_1$ {\it via} ${\rm Mod}\mbox{-}{\cal T}_1^{\rm c} \simeq {\cal T}^{\rm c}\mbox{-}{\rm Mod}$.  This situation is considered in \cite[\S 7]{GarkPre2}.  In particular, if ${\cal T}$ is the derived category of modules over a ring then this is so, \cite[7.5]{GarkPre2}, see also \cite{AngHrbParam}; more generally it is so if ${\cal T}$ is an algebraic triangulated category, \cite{BirWilDualPr}. 

\vspace{4pt}

\noindent {\bf Question:}  If ${\cal T}$ is a compactly generated triangulated category, is there a triangulated category ${\cal T}_1$ and an elementary duality between ${\cal T}$ and ${\cal T}_1$?  If such a category ${\cal T}_1$ exists, is it essentially unique?

By ``an elementary duality" we mean at least a natural bijection between definable subcategories, probably also an anti-equivalence between the respective categories of pp-sorts, perhaps also a duality at the level of pp formulas.  See the remarks in Section \ref{secenhanultra} about enhancements.

\vspace{4pt}

This also raises the further general questions.

\vspace{4pt}

\noindent {\bf Questions:}  What is a characterisation of the categories which arise as ${\cal T}^{\rm c}$ where ${\cal T}$ is compactly generated triangulated?  Given such a category, does it come from a unique compactly generated triangulated category ${\cal T}$? and, if so, how can ${\cal T}$ be constructed from it?  In particular is $({\cal T}^{\rm c})^{\rm op}$ of the form ${\cal T}_1^{\rm c}$ for some compactly generated triangulated category ${\cal T}_1$?

These seem to be hard questions to answer; they include the, only partly resolved, Margolis Conjecture in the case that ${\cal T}$ is the stable homotopy category of spectra.

\vspace{4pt}

If ${\cal T}$ is the derived category ${\cal D}_R = {\cal D}({\rm Mod}\mbox{-}R)$ of some ring $R$, we do get a good elementary duality between ${\cal D}_R$ and ${\cal D}_{R^{\rm op}} = {\cal D}(R\mbox{-}{\rm Mod})$.  This follows because the duality $({\rm proj}\mbox{-}R)^{\rm op} \to {\rm proj}\mbox{-}R^{\rm op}$ between the categories of finitely generated projectives given by $P \mapsto (P,R)$ extends to the respective categories of perfect complexes, that is, to a duality $(-)^{\rm t}:({\cal D}_R^{\rm c})^{\rm op} \simeq {\cal D}_{R^{\rm op}}^{\rm c}$, see \cite[\S 7]{GarkPre2}, also \cite[\S 2.2]{AngHrbParam}. In these papers, $R$ is a 1-sorted ring but the arguments also apply if $R$ is a skeletally small preadditive category.  In \cite[\S 3.2]{BirWilDualPr} this is extended to algebraic triangulated categories {\it via} dg-enhancements.  We will, in Section \ref{secintdual}, describe an internal duality, from \cite[Chpt.~7]{WagThesis} in the tensor-triangulated case.  If $R$ is commutative, so ${\cal D}_R \simeq {\cal D}_{R^{\rm op}}$, the duality in \cite{AngHrbParam} does coincide (\cite[7.3.5]{WagThesis}) with the internal duality described in Section \ref{secintdual}.

For details, we refer the reader to those papers; in particular, the generalisation in \cite{BirWilDualPr} to algebraic triangulated categories uses enhancements (see Section \ref{secenhanultra}), which we don't go into here (also see \cite{LakVit} for related use of enhancements).  For an abstract approach to dualities between triangulated categories, see \cite{BirWilDualPr}.

\vspace{4pt}

We continue a little further in the case that ${\cal T}$ is the derived category ${\cal D}_R$ of a module category.  If ${\cal D}$ is a definable subcategory of ${\cal D}_R$, then we have the corresponding annihilator ideal ${\rm Ann}_{{\cal D}_R^{\rm c}}({\cal D})$.  Set $({\rm Ann}_{{\cal D}_R^{\rm c}}({\cal D}))^{\rm t} = \{ f^{\rm t}: f\in {\rm Ann}_{{\cal D}_R^{\rm c}}({\cal D}) \}$, where $(-)^{\rm t}:({\cal D}_R^{\rm c})^{\rm op} \simeq {\cal D}_{R^{\rm op}}^{\rm c}$ is the duality from the previous paragraph.  Then, \cite[2.3]{AngHrbParam}, $({\rm Ann}_{{\cal D}_R^{\rm c}}({\cal D}))^{\rm t}$ is an annihilator ideal of ${\cal D}_{R^{\rm op}}^{\rm c}$.  We set ${\cal D}^{\rm d} = {\rm Ann}_{{\cal D}_{R^{\rm op}}}(({\rm Ann}_{{\cal D}_R^{\rm c}}({\cal D}))^{\rm t})$ and refer to this as the definable subcategory of ${\cal D}_{R^{\rm op}}$ {\bf elementary dual} to ${\cal D}$.  The terminology is further justified by the following, which refers, using the obvious notations, to the other ways of specifying definable subcategories.

\begin{prop}\label{extdualdefn} \marginpar{extdualdefn} (\cite[2.2-2.5]{AngHrbParam}) If ${\cal D}$ is a definable subcategory of ${\cal D}_R$ and ${\cal D}^{\rm d}$ is its elementary dual definable subcategory of ${\cal D}_{R^{\rm op}}$, then:

${\rm Ann}_{{\cal D}_R^{\rm c}}({\cal D}^{\rm t}) = ({\rm Ann}_{{\cal D}_R^{\rm c}}({\cal D}))^{\rm t}$

${\rm Div}_{{\cal D}_R^{\rm c}}({\cal D}^{\rm t}) = ({\cal D}\mbox{-}{\rm TF})^{\rm t}$

${\cal D}^{\rm t} \mbox{-}{\rm TF} = ({\rm Div}_{{\cal D}_R^{\rm c}}({\cal D}))^{\rm t}$.
\end{prop}
\begin{proof} The first is by definition and \cite[2.3]{AngHrbParam}.  For the others consider $f\in {\rm Ann}_{{\cal D}_R^{\rm c}}({\cal D})$ and form the extended triangle
$$\Sigma^{-1}B \xrightarrow{\Sigma^{-1}g} \Sigma^{-1}C \xrightarrow{\Sigma^{-1}h} A \xrightarrow{f} B \xrightarrow{g} C \xrightarrow{h} \Sigma A$$
then dualise it:
$$  (\Sigma A)^{\rm t} = \Sigma^{-1}A^{\rm t} \xrightarrow{h^{\rm t}} C^{\rm t}  \xrightarrow{g^{\rm t}} B^{\rm t} \xrightarrow{f^{\rm t} } A^{\rm t} \xrightarrow{\Sigma \, h^{\rm t} } \Sigma \, C^v \xrightarrow{\Sigma \, g^{\rm t}} \Sigma \, B^{\rm t} .$$
Then we use the equivalences (\ref{three}) from Section \ref{secdefsub}, namely:
$$Xf=0 \quad \mbox{ iff } \quad g|X \quad \mbox{ iff } \quad {\rm ann}_X(\Sigma^{-1}h)=0.$$
From that we directly obtain the other two equalities.
\end{proof}

We also have, just as for definable subcategories of module categories, that the category of pp-pairs for ${\cal D}^{\rm d}$ is the opposite to that for ${\cal D}$.  The latter is equivalent to ${\rm mod}\mbox{-}{\cal D}_R^{\rm c}/{\cal S}_{\cal D}$, where ${\cal S}_{\cal D} = \{ G: (G,yX) =0 \,\, \forall X\in {\cal D} \}$.  We set $d{\cal S}_{\cal D} = \{ dG: G\in {\cal S}_{\cal D}\}$, where $d$ is the duality of \ref{rtlmods}.\footnote{One can set up duality at the level of pp formulas but it's duality of pp-pairs which we really need.  Also see Section \ref{secintdual} for the issues re well-definedness/independence of enhancements which arise.}

\begin{prop}\label{elemdualimag} \marginpar{elemdualimag} If ${\cal D}$ is a definable subcategory of ${\cal D}_R$ and ${\cal D}^{\rm d}$ is its elementary dual definable subcategory of ${\cal D}_{R^{\rm op}}$, then 
$${\cal S}_{{\cal D}^{\rm d}} = d{\cal S}_{\cal D}.$$  
Hence 
$${\mathbb L}^{\rm eq+}({\cal D}^{\rm d}) = ({\cal D}^{\rm c}_{R^{\rm op}})\mbox{-}{\rm mod} /{\cal S}_{{\cal D}^{\rm d}} \simeq ({\rm mod}\mbox{-}{\cal D}_R^{\rm c}/{\cal S}_{\cal D})^{\rm op} = ({\mathbb L}^{\rm eq+}({\cal D}))^{\rm op}.$$
\end{prop}

This is a special case of \cite[7.4]{GarkPre2} which deals with the general case of pairs, ${\cal T}$, ${\cal T}_1$, of compactly generated triangulated categories with ${\cal T}_1^{\rm c} \simeq ({\cal T}^{\rm c})^{\rm op}$, also showing that, in this situation, we have a frame isomorphism between ${\rm Zg}({\cal T})$ and ${\rm Zg}({\cal T}_1)$.

It is shown in \cite{AngHrbParam} that, for derived categories of module categories, elementary duality has the same relation to algebraic Hom-dualities as in the case of definable subcategories of module categories. In \cite{BirWilDualPr} this is treated in a very general way and a variety of specific examples, from algebra and topology, are given.

\section{Tensor-triangulated categories} \label{secttcats} \marginpar{secttcats}

Suppose now that the compactly generated triangulated category ${\cal T}$ has a monoidal, that is a tensor, structure.  So we have $\otimes:{\cal T} \times {\cal T} \to {\cal T}$, which we assume to be commutative as well as associative, for which we have a tensor-unit $\mathbbm{1}$ - so $\mathbbm{1} \otimes X \simeq X$ for every $X\in {\cal T}$.  {\bf We assume $\otimes$ to be exact in each variable.}  We drop explicit mention of associators {\it et cetera}, see for instance \cite[Part II]{Lev} for more background.

We will suppose that ${\cal T}$ is {\bf rigidly-compactly generated}.  That is, we assume in addition:  

\noindent that the tensor structure is {\bf closed}, meaning that there is an internal hom $[-,-]:{\cal T}\times {\cal T} \to {\cal T}$ which is right adjoint to $\otimes$:  $(X\otimes Y,Z) \simeq (X,[Y,Z])$ for $X,Y,Z \in {\cal T}$, in particular $(Y,Z) \simeq (\mathbbm{1}, [Y,Z])$;  

\noindent and, writing $X^\vee = [X,\mathbbm{1}]$ for the {\bf dual} of an object $X\in {\cal T}$, we assume that every compact object $A$ is {\bf rigid}, meaning that the natural map $A^\vee \otimes B \to [A,B]$ is an isomorphism for every $B\in {\cal T}^{\rm c}$.  

\noindent It follows that ${\cal T}^{\rm c}$ is a {\bf tensor-subcategory} of ${\cal T}$ (i.e.~is closed under $\otimes$), that $(A^\vee)^\vee \simeq A$, that $A^\vee \otimes X \simeq [A,X]$ for $X\in {\cal T}$ and $A\in {\cal T}^{\rm c}$, and that the duality functor $(-)^\vee$ is exact (e.g.~see \cite[\S 1, 2.12]{StevTour}).

The monoidal structure on ${\cal T}^{\rm c}$ induces, by Day convolution (see \cite[Appx.]{BKSFrame}), a right-exact monoidal structure on ${\rm mod}\mbox{-}{\cal T}^{\rm c}$ and hence on ${\rm Mod}\mbox{-}{\cal T}^{\rm c}$.  By definition we have $y(A\otimes B) \simeq yA \otimes yB$ for $A, B \in {\cal T}^{\rm c}$ and, see \cite[A.14]{BKSFrame}, the restricted Yoneda functor $y:{\cal T} \to {\rm Mod}\mbox{-}{\cal T}^{\rm c}$ is monoidal.  The duality \ref{cohfpdual} between ${\rm mod}\mbox{-}{\cal T}^{\rm op}$ and ${\rm Coh}({\cal T})$ is monoidal if the latter is given the natural tensor structure (see \cite[\S 5.1]{WagThesis}).

\vspace{4pt}

We say that a definable subcategory ${\cal D}$ of ${\cal T}$ is {\bf tensor-closed} if, for every $X\in {\cal D}$ and $Y\in {\cal T}$, we have $X\otimes Y \in {\cal D}$.  It is sufficient, see below, that this be so for every $Y\in {\cal T}^{\rm c}$.  The theorem below says that this tensor-closed condition is equivalent to corresponding requirements on the associated data.  We write $f\otimes A$ for $f\otimes {\rm id}_A$ if $f$ is a morphism and $A$ an object.

\begin{theorem}\label{tenscorr} \marginpar{tenscorr} \cite[5.1.8]{WagThesis} Suppose that ${\cal T}$ is a rigidly-compactly generated tensor-triangulated category.  Then the following conditions on a definable subcategory ${\cal D}$ are equivalent:

\noindent (i) ${\cal D}$ is tensor-closed;

\noindent (ii) $X\in {\cal D}$ and $A \in {\cal T}^{\rm c}$ implies $X\otimes A \in {\cal D}$;

\noindent (iii) if $f\in {\rm Ann}_{{\cal T}^{\rm c}}({\cal D})$ and $A\in {\cal T}^{\rm c}$, then $f\otimes A \in {\rm Ann}_{{\cal T}^{\rm c}}({\cal D})$;

\noindent (iv) the corresponding Serre subcategory ${\cal S}_{\cal D}$ of ${\rm mod}\mbox{-}{\cal T}^{\rm c}$ is a tensor-ideal of ${\rm mod}\mbox{-}{\cal T}^{\rm c}$ (it is enough that it be closed under tensoring with representable functors $yA$ with $A\in {\cal T}^{\rm c}$);

\noindent (v) the corresponding Serre subcategory ${\rm Ann}_{{\rm Coh}({\cal T})}({\cal D}) = {\cal S}_{\cal D}^\circ$ of ${\rm Coh}({\cal T})$ is a tensor-ideal of ${\rm Coh}({\cal T})$ (it is enough that it be closed under tensoring with representable functors $(A,-)$ with $A\in {\cal T}^{\rm c}$).
\end{theorem}

A stronger condition on a definable subcategory ${\cal D}$ of ${\cal T}$ is that it be a {\bf tensor-ideal} of ${\cal T}$, meaning that it is tensor-closed and triangulated.  The corresponding, in the sense of \ref{tenscorr}, annihilator ideals and Serre subcategories are characterised in \cite[5.2.14]{WagThesis}.  The additional condition on ${\rm Ann}_{{\cal T}^{\rm c}}({\cal D})$ is that it be exact and the additional condition on ${\cal S}_{\cal D}$ is that it be perfect; these conditions come from \cite{KraCohQuot}, see \cite[\S 5.2]{WagThesis} for the detailed statements.  Furthermore, the tensor version of \ref{smash1} is true:  the triangulated tensor-closed definable subcategories of ${\cal T}$ are in bijection, {\it via} torsion pairs, with the smashing tensor-ideals of ${\cal T}$ ( \cite[5.2.14]{WagThesis}).

In \cite[Chpt.~6]{WagThesis}, Wagstaffe defines and investigates various coarsenings of the Ziegler topology on ${\rm pinj}({\cal T})$, in particular, the tensor-closed Ziegler spectrum, ${\rm Zg}^\otimes({\cal T})$, which is obtained by taking the closed subsets to be those of the form ${\cal D} \cap {\rm pinj}({\cal T})$ where ${\cal D}$ is a tensor-closed definable subcategory of ${\cal T}$.

\subsection{Spectra in tensor-triangulated categories}\label{secttspec} \marginpar{secttspec}

A {\bf prime} of the tensor-triangulated category ${\cal T}$ is a (thick) tensor-ideal ${\cal P}$ of ${\cal T}^{\rm c}$ such that if $A,B\in {\cal T}^{\rm c}$ and $A\otimes B\in {\cal P}$, then $A$ or $B$ is in ${\cal P}$.  The {\bf Balmer spectrum} \cite{BalSpec}, ${\rm Spc}({\cal T}^{\rm c})$ or just ${\rm Spc}({\cal T})$, consists of these primes, with the topology which has, for a basis of open sets, those of the form
$$U(A) = \{ {\cal P}\in {\rm Spc}({\cal T}): A\in {\cal P}\}$$
for $A\in {\cal T}^{\rm c}$.  This is a spectral space and we may also consider, as in Section \ref{secspec}, the Hochster-dual, or {\bf Thomason}, topology on the same set, which is defined by declaring that the $U(A)$ generate, under finite union and arbitrary intersection, the {\em closed} sets.  Both these topologies are natural and have their uses in various contexts, see, for instance, \cite{BalGuide}.

There are various routes by which ${\rm Spc}({\cal T})$ and ${\rm inj}\mbox{-}{\cal T}^{\rm c}$, and also the homological spectrum, ${\rm Spc}^{\rm h}({\cal T})$, from \cite{BalNilp}, with their various topologies, may be connected, see in particular \cite{BirWilHomSpec} and references therein.  We also have the following.

To a point ${\cal P}$ of ${\rm Spc}({\cal T})$ we can associate the finite type hereditary torsion theory $\gamma_{\cal P} = (\overrightarrow{{\cal S}_{y{\cal P}}}, (y{\cal P})^\perp)$ on ${\rm Mod}\mbox{-}{\cal T}^{\rm c}$(see Section \ref{sectors}) whose torsion class is generated as such by $y{\cal P}$, that is, the torsion class is the $\varinjlim$-closure of the Serre subcategory ${\cal S}_{y{\cal P}}$ generated by $y{\cal P}$.  

By \cite[3.9]{BalNilp} this gives an injection of the lattice of Balmer primes into the lattice of finite-type hereditary torsion theories, the latter ordered by inclusion of torsion classes.  For, if ${\cal P} \subset {\cal Q}$ is a proper inclusion of Balmer primes, then, by Balmer's result, there is a maximal Serre tensor-ideal ${\cal B}$ of ${\rm mod}\mbox{-}{\cal T}^{\rm c}$ such that ${\cal P} = y^{-1}{\cal B}$.  Certainly ${\cal S}_{y{\cal P}} \subseteq {\cal B}$ so, if we had ${\cal S}_{y{\cal P}} = {\cal S}_{y{\cal Q}}$, then we would have $y{\cal Q} \subseteq {\cal B}$ and hence a contradiction.

Further, each finite type hereditary torsionfree class ${\cal F}$ is determined by its intersection with ${\rm inj}_{{\cal T}^{\rm c}}$, see \cite[11.1.29]{PreNBK}, and the resulting sets ${\cal F} \, \cap \, {\rm inj}_{{\cal T}^{\rm c}}$ are the closed sets in the Ziegler topology on ${\rm inj}_{{\cal T}^{\rm c}}$ (see \cite[\S 14.1.3]{PreNBK}).  So, to a Balmer prime ${\cal P}$, we also have the associated Ziegler-closed set $(y{\cal P})^\perp \, \cap \, {\rm inj}_{{\cal T}^{\rm c}}$.  Note that this association is inclusion-reversing.

If $A\in {\cal T}^{\rm c}$ then we have 
$${\cal P} \in U(A) \text{ iff } A\in {\cal P} \text{ iff } yA \in \overrightarrow{{\cal S}_{y{\cal P}}} \text{ iff } (y{\cal P})^\perp \subseteq (yA)^\perp.$$
The second equivalence is by the argument just made.  Note that $(yA)^\perp \, \cap \, {\rm inj}_{{\cal T}^{\rm c}}$ is the complement of the basic Ziegler-open subset of ${\rm inj}_{{\cal T}^{\rm c}}$ that is defined by $(yA,-) \neq 0$, hence it is basic open in the dual-Ziegler topology.

For instance, if $R$ is commutative noetherian, then the above essentially gives the embedding (see \cite{BalSpec}, \cite{GarkPre3}) of ${\rm Spc}({\cal D}_R^{\rm perf})$ with the Thomason topology into the frame of Ziegler-open subsets of ${\rm Spec}(R)$, the latter being isomorphic, as a lattice, to the opposite of the lattice of finite type hereditary torsionfree classes of $R$-modules.

\subsection{Internal duality in tensor-triangulated categories}\label{secintdual} \marginpar{secintdual}

In \cite[Chpt.~7]{WagThesis} an {\em internal} duality for rigidly-compactly generated tensor-triangulated categories ${\cal T}$ is defined.  In this respect it is somewhat similar to elementary duality in the case that $R$ is a commutative ring, since then the categories of right and left $R$-modules are naturally identified and so, in that particular context, elementary duality is an internal duality on ${\rm Mod}\mbox{-}R$.  Indeed, for a commutative ring $R$ and the derived-tensor structure on the derived category ${\cal D}_R$, this internal duality coincides with elementary duality, \cite[7.3.5]{WagThesis}.

The internal duality for rigidly-compactly generated tensor-triangulated ${\cal T}$ comes from the second author's thesis \cite{WagThesis} and it was also discovered independently by Bird and Williamson \cite{BirWilDualPr}.  In \cite{WagThesis} it is defined in terms of cohomological ideals, Serre subcategories and definable subcategories; here we note that it can also be defined at the level of formulas and pp-pairs.  We continue to assume that ${\cal T}$ is a rigidly-compactly generated tensor-triangulated category.

\vspace{4pt}

Just as for the ``external" duality, we can define the duality using a hom functor to an object but, in this case, we use the internal hom functor:  for $A\in {\cal T}^{\rm c}$, consider $A \mapsto [A,\mathbbm 1] \simeq A^\vee \otimes {\mathbbm 1} \simeq A^\vee$.  Similarly, internal duality $(-)^\vee = [-,{\mathbbm 1}]$ applied to a morphism $A \xrightarrow{f} B$ in ${\cal T}^{\rm c}$ gives the morphism $B^\vee \xrightarrow{f^\vee} A^\vee$ in ${\cal T}^{\rm c}$.  Since ${\cal T}$ is rigidly-compactly generated, we have that $(-)^\vee$ is an anti-equivalence $({\cal T}^{\rm c})^{\rm op} \simeq {\cal T}^{\rm c}$ with ${(-)^\vee}^\vee$ naturally equivalent to the identity functor on ${\cal T}^{\rm c}$ (see \cite[1.4]{StevTour}).  We also apply these notations to arbitary objects and morphisms of ${\cal T}$.

Given a definable subcategory ${\cal D}$ of ${\cal T}$, with associated annihilator ideal ${\rm Ann}_{{\cal T}^{\rm c}}({\cal D})$, we define its {\bf internal dual} definable subcategory of ${\cal T}$ to be ${\cal D}^\vee = {\rm Ann}_{\cal T}({\rm Ann}_{{\cal T}^{\rm c}}({\cal D})^\vee)$, where we set ${\cal A}^\vee = \{ f^\vee: f\in {\cal A}\}$ for ${\cal A}$ a collection of morphisms in ${\cal T}^{\rm c}$.

\begin{prop}\label{intdualdefn} \marginpar{intdualdefn} (mostly \cite[\S 7.1]{WagThesis}) Suppose that ${\cal T}$ is a rigidly-compactly generated tensor-triangulated category, let ${\cal D}$ be a definable subcategory and consider its elementary dual definable subcategory ${\cal D}^\vee$.  Then $({\rm Ann}_{{\cal T}^{\rm c}}({\cal D}))^\vee$ is an annihilator ideal, $({\cal D}^\vee)^\vee = {\cal D}$ and we have the following:

${\rm Ann}_{{\cal T}^{\rm c}}({\cal D}^\vee) = ({\rm Ann}_{{\cal T}^{\rm c}}({\cal D})^\vee)$

${\rm Div}_{{\cal T}^{\rm c}}({\cal D}^\vee) = ({\cal D}\mbox{-}{\rm TF})^\vee$

${\cal D}^\vee \mbox{-}{\rm TF} = ({\rm Div}_{{\cal T}^{\rm c}}({\cal D}))^\vee$.
\end{prop}
\begin{proof} The proof is very similar to that of \ref{extdualdefn}, using \cite[\S 7]{GarkPre2} to get the first statements.  For the last two, consider $f\in {\rm Ann}_{{\cal T}^{\rm c}}({\cal D})$ and form the extended triangle
$$\Sigma^{-1}B \xrightarrow{\Sigma^{-1}g} \Sigma^{-1}C \xrightarrow{\Sigma^{-1}h} A \xrightarrow{f} B \xrightarrow{g} C \xrightarrow{h} \Sigma A$$
then dualise it:
$$  (\Sigma A)^\vee = \Sigma^{-1}A^\vee \xrightarrow{h^\vee} C^\vee  \xrightarrow{g^\vee} B^\vee \xrightarrow{f^\vee } A^\vee \xrightarrow{\Sigma \, h^\vee } \Sigma \, C^\vee \xrightarrow{\Sigma \, g^\vee} \Sigma \, B^\vee .$$
Then apply Equation (\ref{three}) from Section \ref{secdefsub}.
\end{proof}

This internal duality can also be given by a duality operation on pp formulas and pp-pairs.  This is defined exactly as one would expect from the abelian/modules case.  Namely, if $\phi(x_B)$, being $\exists x_{B'} \, (x_Bf=x_{B'}f')$, is a typical pp formula, where $f:A\to B$ and $f':A\to B'$ are in ${\cal T}^{\rm c}$, then we define the {\bf dual} pp formula, $\phi^\vee(x_{B^\vee})$ to be $\exists y_{A^\vee} \, (y_{A^\vee}f^\vee =x_{B^\vee} \, \wedge \, y_{A^\vee}f'^\vee =0_{B'^\vee})$.  In particular, the dual of the pp formula $x_Bf=0$, where $f:A \to B$, is $f^\vee | x_{B^\vee}$ and the dual of $f'|x_B$ is $x_{B^\vee}f'^\vee=0$.

The {\bf dual} of a pp-pair $\phi/\psi$ is then defined to be $\psi^\vee/\phi^\vee$.  

Note that what we have defined here is an internal duality on pp formulas in the language for (right) ${\cal T}^{\rm c}$-modules.  There is a subtlety, which is pointed out in \cite{WagThesis}.  Namely, two pp formulas might be equivalent on ${\cal T}$ - that is, have the same solution set on every object of ${\cal T}$ - yet their duals might not be equivalent.  Indeed, we might have pp formulas $\phi$, $\phi_1$ with $\phi(X) = \phi_1(X)$ for every $X\in {\cal T}$, yet with $\phi^\vee(X) \neq \phi_1^\vee(X)$ perhaps even for every $X \in {\cal T}$ since these might be definable subgroups of distinct sorts - see \cite[Example 7.1.4]{WagThesis}.  Nevertheless $\phi^\vee$ and $\phi_1^\vee$ will define isomorphic coherent functors, meaning that the pairs $\phi^\vee(x)/(x=0)$ and $\phi_1^\vee(x_1)/(x_1=0)$ will be isomorphic in the category ${\mathbb L}({\cal T})^{\rm eq+}$ of pp-imaginaries for ${\cal T}$.  More generally, if $\phi/\psi$ is a pp-pair with $\phi_1$ equivalent to $\phi$ and $\psi_1$ equivalent to $\psi$, then the pp-pairs $\psi^\vee/\phi^\vee$ and $\psi_1^\vee/\phi_1^\vee$ might be distinct but they will be isomorphic; in particular for every $X\in {\cal T}$, we will have $\psi^\vee(X)/\phi^\vee(X) =0$ iff $\psi_1^\vee(X)/\phi_1^\vee(X)=0$.  That follows from \cite[7.4]{GarkPre2}, cf.~\ref{elemdualimag}, indeed it follows that there is an induced anti-isomorphism of the category ${\mathbb L}({\cal T})^{\rm eq+}$ with itself.

\vspace{4pt}

We give some more detail; see also \cite[Chpt.~7]{WagThesis}.  Since we have a duality $(-)^\vee: ({\cal T}^{\rm c})^{\rm op} \to {\cal T}^{\rm c}$ we have, by \cite[7.4]{GarkPre2}, an equivalence ${\rm mod}\mbox{-}{\cal T}^{\rm c} \to {\cal T}^{\rm c}\mbox{-}{\rm mod}$ which is given by taking $G_f = {\rm coker}((-,f): (-,A) \to (-,B))$, where $f:A \to B$, to $F_{f^\vee} = {\rm coker}((f^\vee, -): (A^\vee,-) \to (B^\vee, -))$.  We also have the duality ${\cal T}^{\rm c}\mbox{-}{\rm mod} (\simeq {\rm Coh}({\cal T})) \to ({\rm mod}\mbox{-}{\cal T}^{\rm c})^{\rm op}$ which takes $F_{f^\vee}$ to $(F_{f^\vee})^\diamond: C \mapsto (F_{f^\vee}, (C,-))$ for $C\in {\cal T}^{\rm c}$.

Composing these, we have a duality ${\rm mod}\mbox{-}{\cal T}^{\rm c} \to {\rm mod}\mbox{-}{\cal T}^{\rm c}$ which takes $G_f$ to $(F_{f^\vee})^\diamond$.  In view of the exact sequence (\ref{eq3})
$$0 \to (F_{f^\vee})^\diamond \to (-, B^\vee) \xrightarrow{(-,f^\vee)} (-, A^\vee) \to G_{f^\vee} \to 0$$
we can formulate this as follows.

\begin{prop} Suppose that ${\cal T}$ is a rigidly-compactly generated tensor-triangulated category.  Then there is a duality on ${\rm mod}\mbox{-}{\cal T}^{\rm c}$ which is given on objects by $G_f \mapsto {\rm ker}(-,f^\vee)$, where $(-)^\vee$ is the duality on ${\cal T}^{\rm c}$.
\end{prop}

The next result follows directly from \cite[6.12]{BirWilDualPr} (also \cite[2.3]{AngHrbParam} in the case ${\cal T} ={\cal D}_R$, $R$ commutative).

\begin{prop} Suppose that ${\cal T}$ is a rigidly-compactly generated tensor-triangulated category and let ${\cal D}$ be a definable subcategory.  Then the definable subcategory of ${\cal T}$ generated by the collection of objects $\{ X^\vee: X\in {\cal D}\}$ is exactly the dual definable subcategory ${\cal D}^\vee$.
\end{prop}

There is potential ambiguity in the notation ${\cal D}^\vee$ - we have defined it to be the dual definable subcategory but it would also be a natural notation for $\{ X^\vee: X\in {\cal D}\}$ but the latter, a subclass of ${\cal D}^\vee$, is not in general all of the definable category ${\cal D}^\vee$ (it might not be closed under pure subobjects).

\vspace{4pt}

Tensor-closed definable subcategories are self-dual.

\begin{theorem}\label{tenscloseddual} \marginpar{tenscloseddual} \cite[7.2.2]{WagThesis} If ${\cal D}$ is a tensor-closed definable subcategory of a rigidly-compactly generated tensor-triangulated category, then ${\cal D}$ is self-dual: ${\cal D}^\vee = {\cal D}$.
\end{theorem}

\subsection{Internal Hom interpretation} \label{secinthom} \marginpar{secinthom}

We finish by pointing out some more ideals of ${\cal T}^{\rm c}$ associated to a definable category ${\cal D}$ in the rigidly-compactly generated tensor-triangulated context.  They appear (along with their rather provisional names) in the statement of the next result.

\begin{prop}\label{4ideals} \marginpar{4ideals}  Suppose that ${\cal T}$ is a rigidly-compactly generated tensor-triangulated category and let ${\cal X}\subset {\cal T}$.  We define the {\bf tensor-annihilator} of ${\cal X}$:
$$\otimes\mbox{-}{\rm ann}_{{\cal T}^{\rm c}}{\cal X} = \{ f:a\to b \in {\cal T}^{\rm c}: f\otimes X=0: a\otimes X \to b\otimes X \,\, \forall X\in {\cal X} \},
$$
the {\bf internal-hom-annihilator} of ${\cal X}$:
$$[{\rm ann}]_{{\cal T}^{\rm c}}{\cal X} = \{ f:a\to b \in {\cal T}^{\rm c}: [f,X]=0: [b,X] \to [a,X] \,\, \forall X\in {\cal X} \},
$$
the {\bf tensor phantomiser} of ${\cal X}$:
$$\otimes\mbox{-}{\rm phan}_{{\cal T}^{\rm c}}{\cal X} = \{ f:a\to b \in {\cal T}^{\rm c}: f\otimes X: a\otimes X \to b\otimes X \text{ is phantom } \forall X\in {\cal X} \},
$$
and the {\bf internal-hom-phantomiser} of ${\cal X}$:
$$[{\rm phan}]_{{\cal T}^{\rm c}}{\cal X} = \{ f:a\to b \in {\cal T}^{\rm c}: [f,X]: [b,X] \to [a,X] \text{ is phantom } \forall X\in {\cal X} \}.
$$
All these are ideals of ${\cal T}^{\rm c}$ and the tensor-annihilator and internal-hom-annihilator are dual ideals:  
$$(\otimes\mbox{-}{\rm ann}_{{\cal T}^{\rm c}}{\cal X})^\vee = [{\rm ann}]_{{\cal T}^{\rm c}}{\cal X}.
$$
Moreover, the tensor phantomiser and internal-hom-phantomiser coincide (we could call this the {\bf phantomiser}) and are equal to the annihilator ideal of the smallest tensor-closed definable subcategory $\langle {\cal X} \rangle^\otimes$ containing ${\cal X}$:
$$\otimes\mbox{-}{\rm phan}_{{\cal T}^{\rm c}}{\cal X} = [{\rm phan}]_{{\cal T}^{\rm c}}{\cal X} = {\rm Ann}_{{\cal T}^{\rm c}} \,\langle {\cal X} \rangle^\otimes.$$
Therefore this is also the annihilator ideal generated by each of $\otimes\mbox{-}{\rm ann}_{{\cal T}^{\rm c}}{\cal X}$ and $[{\rm ann}]_{{\cal T}^{\rm c}}{\cal X}$.
\end{prop}
\begin{proof}  For every $X\in {\cal T}$, $A\otimes X \xrightarrow{f\otimes X} B\otimes X$ is (isomorphic to) $A^{\vee \vee} \otimes X \xrightarrow{f^{\vee \vee}} B^{\vee \vee}\otimes X$ and therefore is $[A^\vee,X] \xrightarrow{[f^\vee,X]} [B^\vee,X]$.  Thus, the condition $f\otimes X =0: A\otimes X \to B\otimes X$ is equivalent to the condition $[f^\vee,X]=0: [A^\vee,X] \to [B^\vee,X]$ and we have $\otimes\mbox{-}{\rm ann}_{{\cal T}^{\rm c}}{\cal X} = ([{\rm ann}]_{\cal T}{\cal X})^\vee$.

For the other parts, we have $f\in \otimes\mbox{-}{\rm phan}_{{\cal T}^{\rm c}}{\cal X}$ iff for every $c\in {\cal T}^{\rm c}$ we have $(c, f\otimes X)=0$, that is $(f^\vee,c^\vee \otimes X)=0$ which, since every compact object is a dual, is equivalent to $f^\vee \in  \otimes\mbox{-}{\rm ann}_{{\cal T}^{\rm c}}\, \langle{\cal X}\rangle^\otimes$.  By \ref{tenscloseddual}, $f^\vee \in  \otimes\mbox{-}{\rm ann}_{{\cal T}^{\rm c}}\, \langle{\cal X}\rangle^\otimes$ iff $f \in  \otimes\mbox{-}{\rm ann}_{{\cal T}^{\rm c}}\, \langle{\cal X}\rangle^\otimes$.  Therefore $\otimes\mbox{-}{\rm phan}_{{\cal T}^{\rm c}}{\cal X} = \otimes\mbox{-}{\rm ann}_{{\cal T}^{\rm c}}\, \langle{\cal X}\rangle^\otimes$.

Also, $f\in [{\rm phan}]_{{\cal T}^{\rm c}}{\cal X}$ iff for every $c\in {\cal T}^{\rm c}$ we have $(c,[f,X])=0$, equivalently $(f,c^\vee \otimes X)=0$ which, since every compact object is a dual, is equivalent to $f\in {\rm ann}_{{\cal T}^{\rm c}}\, \langle{\cal X}\rangle^\otimes$.  Therefore $[{\rm phan}]_{{\cal T}^{\rm c}}{\cal X} = {\rm ann}_{{\cal T}^{\rm c}}\, \langle{\cal X}\rangle^\otimes = \otimes\mbox{-}{\rm phan}_{{\cal T}^{\rm c}}{\cal X}$, as claimed.
\end{proof}

Note that the condition $f^\vee\in [{\rm ann}]_{\cal T}{\cal X}$ is expressed by the condition ``$Xf^\vee=0$" with $B^\vee \xrightarrow{f^\vee} A^\vee$.  This looks like an annihilator sentence but it is for internal hom, rather than actual hom, groups.
This suggests an alternative, internal-hom, interpretation of the model-theoretic language, remarked at \ref{rmkinthom}, when ${\cal T}$ is a rigidly-compactly generated tensor-triangulated category.  In this interpretation the value of $X\in {\cal T}$ at sort $A\in {\cal T}^{\rm c}$ is $[A,X]$, rather than $(A,X)$, and the interpretation of $A\xrightarrow{f} B \in {\cal T}^{\rm c}$ in $X$ is $[f,X]:[B,X] \to [A,X]$ rather than $(f,X):(B,X) \to (A,X)$.  In this interpretation of the language the values of sorts at objects of ${\cal T}$ are again objects of ${\cal T}$, not abelian groups.

This also constitutes an alternative ``internal restricted Yoneda" functor from ${\cal T}$ to the ``${\cal T}$-valued-module category" ${\rm Mod}_{\cal T}\mbox{-}{\cal T}^{\rm c} = (({\cal T}^{\rm c})^{\rm op}, {\cal T})$, which takes $X\in {\cal T}$ to the functor $[-,X]:({\cal T}^{\rm c})^{\rm op} \to {\cal T}$ and takes $f:X\to Y$ to $[-,f]:[-,X] \to [-,Y]$.  In this internal-hom interpretation, the language for ${\cal T}$ stays the same but the interpretation has changed:  instead of $(-,X)$ we use $[-,X]$.

Similarly, the tensor-annihilator that we defined above belongs to a third (in this case, covariant) interpretation of the same language, based on $-\otimes X$, rather than $(-,X)$ or $[-, X]$.

In both these new interpretations the sorts belong to ${\cal T}$ rather than to ${\bf Ab}$, so we cannot immediately make sense of ``elements" of a sort.  But, using the idea of an ``element" being an arrow from the tensor-unit $\mathbbm{1}$, we can move back to the category of ${\cal T}^{\rm c}$-modules.  If we do that, we recover the usual interpretation (from the internal-hom interpretation) and an `internal dual' interpretation (from the tensor interpretation).  That is, we have:

$$y:{\cal T} \to {\rm Mod}\mbox{-}{\cal T}^{\rm c} \text{ given by } X\mapsto (-,X);$$

$$[y]:{\cal T} \to (({\cal T}^{\rm c})^{\rm op}, {\cal T}) \text{ given by }X\mapsto [-,X];$$

$$\epsilon:{\cal T} \to ({\cal T}^{\rm c}, {\cal T})\text{ given by }X\mapsto (-\otimes X).$$

The latter two can then be composed with $(\mathbbm{1},-)$:
$$(\mathbbm{1},-)[y]=y:{\cal T} \to (({\cal T}^{\rm c})^{\rm op}, {\cal T}) \to {\rm Mod}\mbox{-}{\cal T}^{\rm c} 
$$
$$\text{ given by }X\mapsto [-,X] \mapsto (\mathbbm{1},[-,X]) \simeq (-,X);
$$
and
$$(\mathbbm{1},-)\epsilon:{\cal T} \to ({\cal T}^{\rm c}, {\cal T}) \to {\cal T}^{\rm c}\mbox{-}{\rm Mod}$$
$$\text{ given by } X\mapsto (-\otimes X) \mapsto (\mathbbm{1}, -\otimes X) \simeq (\mathbbm{1},[(-)^\vee, X]) \simeq ((-)^\vee,X)$$ 

\vspace{4pt}

Also, essentially following \cite[4.13]{BirWilHomSpec}, note that if $A\in {\cal T}^{\rm c}$ and $X\in {\cal T}$, then $[A,X] =0$ iff, for all $C \in {\cal T}^{\rm c}$, we have $(C,[A,X])=0$ iff, for all $C\in {\cal T}^{\rm c}$, we have $(C\otimes A,X)=0$.  In particular
$$\{N \in {\rm Zg}({\cal T}): [A,N]=0\} = \bigcap_{C\in {\cal T}^{\rm c}} \{N\in {\rm Zg}({\cal T}): (C\otimes A,N)=0\}$$
is an intersection of Ziegler-closed sets, hence is itself Ziegler-closed.

Furthermore, continuing the above computation, we have $[A,X]=0$ iff, for all $C\in {\cal T}^{\rm c}$, we have $(A \otimes C,X)=0$ iff, for all $C\in {\cal T}^{\rm c}$, we have $(A,[C,X])=0$ iff, for all $C\in {\cal T}^{\rm c}$, we have $(A, C^\vee \otimes X)=0$, iff, for all $C\in {\cal T}^{\rm c}$, we have $(A, C\otimes X)=0$.  So if ${\cal D}$ is the definable subcategory of ${\cal T}$ cut out by the condition $(A,-)=0$, then the condition $[A,-]=0$ cuts out the smallest tensor-closed definable subcategory of ${\cal T}$ containing ${\cal D}$.

\end{document}